\documentclass[11pt]{amsart}
\usepackage[margin=1in]{geometry}
\usepackage{amssymb,amsfonts,amsmath,mathtools}
\usepackage{color}
\usepackage{soul}

\usepackage{enumerate}
\usepackage{mathrsfs}
\usepackage[unicode]{hyperref}
\usepackage[capitalise]{cleveref}
\usepackage{ upgreek }

\usepackage{constants}
\usepackage{bbm}

\newtheorem{theorem}{Theorem}[section]
\newtheorem{corollary}[theorem]{Corollary}
\newtheorem{hypothesis}[theorem]{Hypothesis}
\newtheorem{proposition}[theorem]{Proposition}
\newtheorem{lemma}[theorem]{Lemma}

\newtheorem{question*}{Question}
\newtheorem{problem*}{Problem}

\theoremstyle{definition}

\theoremstyle{remark}
\newtheorem*{remark}{Remark}

\numberwithin{equation}{section}

\crefname{figure}{Figure}{Figures}
\theoremstyle{plain}
\newtheorem*{theorem*}{Theorem}
\crefname{theorems}{Theorem}{Theorems}
\crefname{corollaries}{Corollary}{Corollaries}
\newtheorem*{corollary*}{Corollary}
\crefname{corollaries*}{Corollary}{Corollaries}
\crefname{lemma}{Lemma}{Lemmata}
\crefname{proposition}{Proposition}{Propositions}
\crefname{conjectures}{Conjecture}{Conjectures}
\newtheorem*{conjonjecture*}{Conjecture}
\crefname{conjonjectures*}{Conjecture}{Conjectures}
\crefname{definitions}{Definition}{Definitions}
\crefname{hypotheses}{Hypothesis}{Hypotheses}

\newcommand{\Z}{\mathbb{Z}}
\newcommand{\R}{\mathbb{R}}
\newcommand{\Q}{\mathbb{Q}}

\newcommand{\re}{\textup{Re}}
\newcommand{\im}{\textup{Im}}

\makeatletter
\DeclareFontFamily{U}  {MnSymbolF}{}
\DeclareSymbolFont{symbolsMN}{U}{MnSymbolF}{m}{n}
\SetSymbolFont{symbolsMN}{bold}{U}{MnSymbolF}{b}{n}
\DeclareFontShape{U}{MnSymbolF}{m}{n}{
    <-6>  MnSymbolF5
   <6-7>  MnSymbolF6
   <7-8>  MnSymbolF7
   <8-9>  MnSymbolF8
   <9-10> MnSymbolF9
  <10-12> MnSymbolF10
  <12->   MnSymbolF12}{}
\DeclareFontShape{U}{MnSymbolF}{b}{n}{
    <-6>  MnSymbolF-Bold5
   <6-7>  MnSymbolF-Bold6
   <7-8>  MnSymbolF-Bold7
   <8-9>  MnSymbolF-Bold8
   <9-10> MnSymbolF-Bold9
  <10-12> MnSymbolF-Bold10
  <12->   MnSymbolF-Bold12}{}
\DeclareMathSymbol{\tbigtimes}{\mathop}{symbolsMN}{2}
\newcommand*{\bigtimes}{%
  \DOTSB
  \tbigtimes
  \slimits@ 
}
\makeatother

\renewcommand{\bar}{\overline}

\renewcommand{\epsilon}{\varepsilon}

\renewcommand{\pmod}[1]{\, (\mathrm{mod} {\, #1})}

\DeclareMathAlphabet{\mathpzc}{OT1}{pzc}{m}{it}

\renewcommand{\pmod}[1]{\,(\mathrm{mod}\,\,#1)}

\newconstantfamily{abcon}{symbol=c}

\makeatletter
\let\@wraptoccontribs\wraptoccontribs
\makeatother

\title[Bombieri's log-free density estimate and S{\'a}rk{\"o}zy's theorem]{An explicit version of Bombieri's log-free density estimate and S{\'a}rk{\"o}zy's theorem for shifted primes}
\author{Jesse Thorner}
\address{Department of Mathematics, University of Illinois, Urbana, IL 61801, USA}
\email{jesse.thorner@gmail.com}
\author{Asif Zaman}
\address{Department of Mathematics, University of Toronto, 40 St. George Street, Toronto, ON, Canada M5S 2E4}
\email{zaman@math.toronto.edu}

\begin{document}

\begin{abstract}

We make explicit Bombieri's refinement of Gallagher's log-free ``large sieve density estimate near $\sigma = 1$'' for Dirichlet $L$-functions.   We use this estimate and recent work of Green to prove that if $N\geq 2$ is an integer, $A\subseteq\{1,\ldots,N\}$, and for all primes $p$ no two elements in $A$ differ by $p-1$, then $|A|\ll N^{1-1/10^{18}}$.  This strengthens a theorem of S{\'a}rk{\"o}zy.
\end{abstract}

\maketitle

\section{Introduction and statement of results}

Let $N\geq 2$ be an integer, and let $\delta(N)$ be the relative density of a largest set $A\subseteq\{1,\ldots,N\}$ such that for all primes $p$, no two elements in $A$ differ by $p-1$.  Erd\H{o}s conjectured that $\delta(N) \to 0$ as $N \to \infty$, and S{\'a}rk{\"o}zy \cite{Sarkozy} proved the conjecture in the form $\delta(N)\ll (\log\log N)^{-2+o(1)}$.  Several improvements followed (e.g., Lucier \cite{Lucier} and Ruzsa and Sanders \cite{RuzsaSanders}).  Wang \cite{Wang} proved that there exists an absolute and effectively computable constant $\Cl[abcon]{Wang}>0$ such that $\delta(N)\ll \exp(-\Cr{Wang}(\log N)^{1/3})$.  The proofs in \cite{RuzsaSanders,Wang} use the distribution of zeros of Dirichlet $L$-functions.  Assuming the generalized Riemann hypothesis for Dirichlet $L$-functions (GRH), Ruzsa and Sanders \cite{RuzsaSanders} proved that $\delta(N)\ll \exp(-\Cr{Wang}(\log N)^{1/2})$ and asked whether GRH implies the existence of a power-saving bound on $\delta(N)$.

In a recent tour de force, Green \cite{Green} unconditionally proved that such a power-saving bound exists---that is, there exists a constant $\Cl[abcon]{Green}>0$ such that $\delta(N)\ll N^{-\Cr{Green}}$.  Assuming GRH, Green's arguments simplify to show that any fixed $0<\Cr{Green}<\frac{1}{12}$ is admissible \cite{Green2}, but an unconditional admissible value of $\Cr{Green}$ was not produced.  In this paper, we produce such an admissible value.

\begin{theorem}
	\label{thm:Sarkozy}
If $A\subseteq\{1,\ldots,N\}$ and no prime $p$ satisfies $p-1\in A-A$, then $|A|\ll N^{1-1/10^{18}}$.
\end{theorem}

The new ingredient that we supply in order to deduce \cref{thm:Sarkozy} from the work in \cite{Green} is a certain log-free zero density estimate for Dirichlet $L$-functions.  Such a result was first proved by  Linnik as part of his proof that there exists an absolute and effectively computable constant $b>0$ such the least prime in an arithmetic progression modulo $q$ is at most $q^b$ \cite{Linnik}.  For an integer $q\geq 1$, let $\chi\pmod{q}$ be a Dirichlet character, and let $L(s,\chi)$ be its $L$-function.  For $Q\geq 3$ and $\sigma\geq 0$, define
\begin{equation}
	\label{eqn:prod_L-functions}
	\mathscr{L}(s,Q):=\prod_{1\leq q\leq Q}\prod_{\substack{\chi\pmod{q} \\ \textup{$\chi$ primitive}}}L(s,\chi),\qquad \beta_1(Q):=\max\{\beta\in\R\colon \mathscr{L}(\beta,Q)=0\}
\end{equation}
and
\begin{equation}
\label{eqn:NchiDef}
\begin{aligned}
N(\sigma,Q)&:=\#\{\rho=\beta+i\gamma\colon \beta>\sigma,~|\gamma|\leq Q,~\mathscr{L}(\rho ,Q)=0\},\\
N^*(\sigma,Q)&:=\#\{\rho=\beta+i\gamma \neq\beta_1(Q)\colon \beta>\sigma,~|\gamma|\leq Q,~\mathscr{L}(\rho ,Q)=0\}.
\end{aligned}
\end{equation}
(All sums and counts over $\rho$ weigh the contribution for each $\rho$ with its multiplicity.)  Gallagher \cite{Gallagher} used Tur{\'a}n's lower bound for power sums with a ``pre-sifted'' large sieve inequality (\cref{lem:large_sieve}) to prove that there exist absolute and effectively computable constants $\Cl[abcon]{Gallagher1}>0$ and $\Cl[abcon]{Gallagher2}>0$ such that 
\begin{equation}
\label{eqn:Gallagher}
N(\sigma,Q)\leq \Cr{Gallagher1} Q^{\Cr{Gallagher2}(1-\sigma)},\qquad \sigma\geq 0,\quad Q\geq 3.
\end{equation}

Jutila \cite{Jutila} used Selberg's theory of pseudocharacters to sharpen \eqref{eqn:Gallagher}, proving that for all $\epsilon>0$, there exists an effectively computable constant $c(\epsilon)>0$ such that
\begin{equation}
\label{eqn:Jutila}
N(\sigma,Q)\leq c(\epsilon) Q^{(6+\epsilon)(1-\sigma)},\qquad \sigma\geq 4/5,\quad Q\geq 3.
\end{equation}
The exponent in \eqref{eqn:Jutila} is as strong as the generalized Lindel{\"o}f hypothesis implies.  Another refinement comes from Bombieri \cite{Bombieri2}, who proved that there exist absolute and effectively computable constants $\Cl[abcon]{Bombieri1}>0$ and $\Cl[abcon]{Bombieri2}>0$ such that
\begin{equation}
	\label{eqn:Bombieri}
N^*(\sigma,Q)\leq \Cr{Bombieri1} (1-\beta_1(Q))(\log Q)Q^{\Cr{Bombieri2}(1-\sigma)},\qquad \sigma\geq 0,\quad Q\geq 3.
\end{equation}
If $\beta_1(Q)$ and $\sigma$ are suitably close to $1$ as a function of $Q$, then \eqref{eqn:Bombieri} improves upon \eqref{eqn:Gallagher} and \eqref{eqn:Jutila}.

We prove explicit versions of \eqref{eqn:Gallagher} and \eqref{eqn:Bombieri} that hold for all $Q\geq 3$ and all $\sigma\geq \frac{39}{40}$ (see \cref{cor:GLFZDE} below for versions that permit $\sigma\geq 0$).  These appear to be the first explicit log-free zero density estimates for Dirichlet $L$-functions that do not require $Q$ to be ``sufficiently large'' or $\sigma$ to be ``sufficiently close'' to $1$, where sufficiency is not explicitly quantified.  Such estimates are useful in contexts beyond our proof of \cref{thm:Sarkozy}.

\begin{theorem}
	\label{thm:GLFZDE}
Let $Q\geq 3$ and $\sigma\geq \frac{39}{40}$.  If $N(\sigma,Q)$ and $N^*(\sigma,Q)$ are as in \eqref{eqn:NchiDef}, then
\[
N(\sigma,Q)\leq 10^{88}(10^{421}Q^{99})^{1-\sigma},\qquad N^*(\sigma,Q)\leq  10^{93} \min\{1,(1-\beta_1(Q))\log Q\}(10^{466}Q^{170})^{1-\sigma}.
\]
\end{theorem}

\begin{remark}
If $Q\geq 5$, then $\mathscr{L}(0,Q)=0$ and $\mathscr{L}(\sigma,Q)\neq 0$ for $\sigma\geq 1$.  Thus, we have that $\beta_1(Q)\in[0,1)$.
\end{remark}

The proof of \cref{thm:Sarkozy} is sensitive to the size of the implied constants, so we cannot use \eqref{eqn:Jutila} without making $c(\epsilon)$ explicit first.  We did not try to estimate $c(\epsilon)$ because this requires us to make certain intricate results of Huxley \cite{Huxley} on large values of Dirichlet polynomials completely explicit in all ranges. While we avoided this potentially onerous task for the sake of brevity, we encourage the interested reader to make \eqref{eqn:Jutila} explicit for all $Q\geq 3$ and all $\sigma\geq\frac{4}{5}$. 	Our proof of \cref{thm:Sarkozy} shows that further improvements to \cref{thm:GLFZDE} do not substantially improve the power savings in \cref{thm:Sarkozy} without a numerical optimization of the arguments in \cite{Green} or a substantial numerical improvement in the standard zero-free region for $\mathscr{L}(s,Q)$.

\subsection*{Outline of the paper}

\cref{sec:zeros} provides auxiliary results on the distribution of zeros of Dirichlet $L$-functions.  \cref{sec:notation} contains notation and conventions and a discussion the central optimization problem in \eqref{eqn:minimization_problem} whose solution is one of two key results (along with \cref{cor:Linnik_lemma}) that determines the quality of the constants in \cref{thm:GLFZDE}.  In \cref{sec:zero_detection}, we prove our zero detection result for zeros of Dirichlet $L$-functions (\cref{thm:detect}) that serves as the basis for our proof of \cref{thm:GLFZDE} in \cref{sec:proof_main_theorem}.  We use \cref{thm:GLFZDE} and the work in \cite{Green} to prove \cref{thm:Sarkozy} in \cref{sec:Sarkozy}.

\subsection*{Acknowledgements}

We thank Ben Green, Peter Humphries, Nathan Ng, Sarah Peluse, and Joni Ter{\"a}v{\"a}inen for helpful conversations and the anonymous referee for helpful comments.  We performed all numerical calculations with Mathematica 12.

\section{Auxiliary results on Dirichlet $L$-functions}
\label{sec:zeros}

Let $\chi\pmod{q_{\chi}}$ be a primitive Dirichlet character, and let
\[
L(s,\chi)=\prod_{p}(1-\chi(p)p^{-s})^{-1}=\sum_{n=1}^{\infty}\frac{\chi(n)}{n^s},\qquad \re(s)>1
\]
be its $L$-function.  We let $\delta_{\chi}$ be the indicator function for the trivial character modulo 1, whose $L$-function is the Riemann zeta function $\zeta(s)$.
\subsection{The Hadamard factorization}
All of the content in this subsection can be found in \cite{Davenport}. Let
\[
\mathfrak{a}_{\chi}=\frac{1-\chi(-1)}{2},\qquad \tau(\chi) = \sum_{j=1}^{q_{\chi}}\chi(n)e^{2\pi i j/q_{\chi}}.
\]
Since $\chi$ is primitive, we have that $|\tau(\chi)|=\sqrt{q_{\chi}}$.  The completed $L$-function
\[
\Lambda(s,\chi)=(s(s-1))^{\delta_{\chi}}\Big(\frac{q_{\chi}}{\pi}\Big)^{\frac{s+\mathfrak{a}_{\chi}}{2}}\Gamma\Big(\frac{s+\mathfrak{a}_{\chi}}{2}\Big)L(s,\chi)
\]
is entire of order one and satisfies the functional equation
\begin{equation}
	\label{eqn:functional_equation}
	\Lambda(s,\chi) = \frac{\tau(\chi)}{i^{\mathfrak{a}_{\chi}}\sqrt{q_{\chi}}}\Lambda(1-s,\bar{\chi}).
\end{equation}
Since $\Lambda(s,\chi)$ is entire of order one, it exhibits a Hadamard factorization
\[
\Lambda(s,\chi) = e^{A_{\chi}+B_{\chi}s}\prod_{\Lambda(\rho ,\chi)=0}\Big(1-\frac{s}{\rho }\Big)e^{\frac{s}{\rho }}.
\]
The zeros $\rho =\beta +i\gamma $ are the nontrivial zeros of $L(s,\chi)$, and they satisfy $0\leq \beta \leq 1$ and $\gamma\in\R$.  The poles of $s^{\delta_{\chi}}\Gamma(\frac{s+\mathfrak{a}_{\chi}}{2})$ are the trivial zeros; these form the sequence $(-2(j+\delta_{\chi})-\mathfrak{a}_{\chi})_{j=0}^{\infty}$.  Upon taking logarithmic derivatives, we find that
\begin{equation}
	\label{eqn:Hadamard}
\frac{L'}{L}(s,\chi) = B_{\chi}+\sum_{\rho }\Big(\frac{1}{s-\rho }+\frac{1}{\rho }\Big)-\delta_{\chi}\Big(\frac{1}{s-1}+\frac{1}{s}\Big)-\frac{1}{2}\log\frac{q_{\chi}}{\pi}-\frac{1}{2}\frac{\Gamma'}{\Gamma}\Big(\frac{s+\mathfrak{a}_{\chi}}{2}\Big).
\end{equation}
Since
\[
\re(B_{\chi})=-\sum_{\rho }\re\Big(\frac{1}{\rho }\Big),
\]
it follows that
\begin{equation}
	\label{eqn:Hadamard_real}
\re\Big(\frac{L'}{L}(s,\chi)\Big) = \sum_{\rho }\re\Big(\frac{1}{s-\rho }\Big)-\delta_{\chi}\re\Big(\frac{1}{s-1}+\frac{1}{s}\Big)-\frac{1}{2}\log\frac{q_{\chi}}{\pi}-\frac{1}{2}\re\Big(\frac{\Gamma'}{\Gamma}\Big(\frac{s+\mathfrak{a}_{\chi}}{2}\Big)\Big).
\end{equation}

\subsection{Zero-free regions}

We assemble some results on zero-free regions for $L(s,\chi)$ (that together will imply \cref{thm:GZFR}) along with a few other useful results.

\begin{lemma}
	\label{lem:RH}
If $\rho=\beta+i\gamma$ is a nontrivial zero of $\zeta(s)$ and $|\gamma|\leq 3\,000\,175\,332\,800$, then $\beta=\frac{1}{2}$.  If $|\gamma|> 3\,000\,175\,332\,800$, then $\beta\leq 1-1/(5.573\,412\log|\gamma|)$.
\end{lemma}
\begin{proof}
	The first part is \cite[Theorem 1]{PlattTrudgian}.  The second part is \cite[Theorem 1]{MossinghoffTrudgian}.
\end{proof}

\begin{lemma}
\label{lem:Platt}
If $q\leq 400\,000$ and $\chi\pmod{q}$ is a nontrivial primitive Dirichlet character, then $L(s,\chi)$ has no real nontrivial zeros and does not vanish in the region
\[
0<\Big|\re(s)-\frac{1}{2}\Big|\leq\frac{1}{2},\qquad 0<|\im(s)|\leq \max\Big\{\frac{10^8}{q},\frac{37\,500\,000(2-\mathfrak{a}_{\chi})}{q}+200\Big\}.
\]
\end{lemma}
\begin{proof}
	This follows from \cite[Theorems 10.1 and 10.2]{Platt}.
\end{proof}

\begin{lemma}
\label{lem:ZFR}
Let $q\geq 3$.  The product $\prod_{\chi\pmod{q}}L(s,\chi)$ has at most one zero in the region
\[
\re(s)\geq 1-\frac{\Cl[abcon]{ZFR}}{\log\max\{q,q|\im(s)|\}},\qquad \Cr{ZFR}=
\frac{1}{9.645\,908\,801}.
\]
If such a zero $\rho_1$ exists, then $q>400\,000$, $\rho_1$ is a real and simple zero of $\mathscr{L}_q(s)$, and there exists a unique quadratic character $\chi_1\pmod{q}$ such that $L(\rho_1,\chi_1)=0$.
\end{lemma}
\begin{proof}
When $q>400\,000$, McCurley \cite[Theorem 1]{McCurley} proved this with $\max\{q,q|\im(s)|\}$ replaced by $\max\{q,q|\im(s)|,10\}$.  The claimed improvements follow from \cref{lem:Platt}.
\end{proof}

\begin{lemma}
\label{lem:Landau}
Let $\chi\pmod{q}$ and $\chi'\pmod{q'}$ be distinct real primitive characters with $\min\{q,q'\}>400\,000$.  If $\beta$ (resp. $\beta'$) is a real zero of $L(s,\chi)$ (resp. $L(s,\chi')$), then
\[
\min\{\beta,\beta'\}<1-\frac{15 - 10 \sqrt{2}}{(5 - \sqrt{5})\log(\frac{q'q}{17})}=1-\frac{0.310\,3782\ldots}{\log(\frac{q'q}{17})}.
\]
\end{lemma}
\begin{proof}
	McCurley \cite[Theorem 2]{McCurley} proved this with $\frac{q'q}{17}$ replaced by $\max\{\frac{q'q}{17},13\}$.  When $\min\{q,q'\}\leq 400\,000$, \cref{lem:Platt} implies that at least one of $L(s,\chi)$ and $L(s,\chi')$ will have no real zero at all.  This means that we may restrict consideration to $q$ and $q'$ such that $\max\{\frac{q'q}{17},13\}=\frac{q'q}{17}$.
\end{proof}

\begin{corollary}[Page]
\label{cor:Page}
Let $Q\geq 3$.  The product
\[
\mathscr{M}(s,Q):=\prod_{q\leq Q}\prod_{\substack{\chi\pmod{q} \\ \textup{$\chi$ primitive} \\ \textup{$\chi$ quadratic}}}L(s,\chi)
\]
has at most one real zero in the interval
\[
s\geq 1-\frac{15 - 10 \sqrt{2}}{2(5 - \sqrt{5})\log Q}=1-\frac{0.155\,1891\ldots}{\log Q}.
\]
If such a real zero, say $\beta_1(Q)$, exists, then $Q> 400\,000$, $\beta_1(Q)$ is a simple zero of $\mathscr{M}(s,Q)$, and there exists $q_1\in(400\,000,Q]$ and a primitive quadratic character $\chi_1\pmod{q_1}$ such that $L(\beta_1(Q),\chi_1)=0$.
\end{corollary}
\begin{proof}
It suffices to assume that $\beta_1(Q)$ exists.  By \cref{lem:Platt}, this ensures that $Q>400\,000$.  Suppose to the contrary that there exist two primitive quadratic characters $\chi\pmod{q}$ and $\chi'\pmod{q'}$ such that $q,q'\leq Q$ and $L(\beta_1(Q),\chi)=L(\beta_1(Q),\chi')=0$.  By \cref{lem:Platt}, we have that $\min\{q,q'\}>400\,000$.  It follows that
\[
\beta_1(Q)\geq 1-\frac{15 - 10 \sqrt{2}}{2(5 - \sqrt{5})\log Q}\geq 1-\frac{15 - 10 \sqrt{2}}{(5 - \sqrt{5})\log(q'q)},
\]
contradicting \cref{lem:Landau}.  Therefore, at most one such character exists.  The simplicity of $\beta_1(Q)$ follows from \cref{lem:ZFR}.
\end{proof}

\begin{theorem}
\label{thm:GZFR}
Let $Q\geq 3$.  Let $\mathscr{L}(s,Q)$ and $\beta_1(Q)$ be as in \eqref{eqn:prod_L-functions}, and let $\Cr{ZFR}$ be as in \cref{lem:ZFR}.  Except possibly at $\beta_1(Q)$, $\mathscr{L}(s,Q)$ is nonzero in the region
\[
\re(s)\geq 1-\frac{\Cr{ZFR}}{\log \max\{Q,Q|\im(s)|\}}.
\]
If $\beta_1(Q)\geq 1-\frac{\Cr{ZFR}}{\log Q}$, then $Q> 400\,000$, $\beta_1(Q)$ is a simple zero of $\mathscr{L}(s,Q)$, there exists one primitive Dirichlet character $\chi_1\pmod{q_1}$ with $q_1\in(400\,000,Q]$ such that $L(\beta_1(Q),\chi_1)=0$, and $\beta_1(Q)\leq 1-100/(\sqrt{q_1}(\log q_1)^2)$.
\end{theorem}
\begin{proof}
The bound on $\beta_1(Q)$ is proved in \cite{Bordignon,Bordignon2}.  \cref{lem:ZFR} and \cref{cor:Page} imply the rest.
\end{proof}

\subsection{Counting zeros}

We use the following estimate for the $\Gamma$-function.

\begin{lemma}
\label{lem:ono_sound}
If $\re(z)\geq\frac{1}{2}$, then
\[
\re\Big(\frac{\Gamma'}{\Gamma}\Big(\frac{z}{2}\Big)\Big)\leq \log|z|-\gamma_{\Q},\qquad \gamma_{\Q}:=\lim_{N\to\infty}\Big(\sum_{n=1}^{N}\frac{1}{n}-\log N\Big)=0.5772156649\ldots
\]
\end{lemma}
\begin{proof}
In \cite[Lemma 4.1]{HoeySatoTate}, it is proved that if $\Gamma_{\R}(z)=\pi^{-z/2}\Gamma(\frac{z}{2})$ and $\re(z)\geq\frac{1}{2}$, then
\[
\re\Big(\frac{\Gamma_{\R}'}{\Gamma_{\R}}(z)\Big)\leq -\frac{\log\pi+\gamma_{\Q}}{2}+\frac{\log|z|}{2}.
\]
Since $\frac{\Gamma_{\R}'}{\Gamma_{\R}}(z)=\frac{1}{2}\frac{\Gamma'}{\Gamma}(\frac{z}{2})-\frac{\log\pi}{2}$, the lemma follows.
\end{proof}

Given $\sigma\geq 0$, $T\geq 0$, $Q\geq 1$, $1\leq q\leq Q$, and a primitive Dirichlet character $\chi\pmod{q}$, we define
\begin{equation}
\label{eqn:def_N_chi(sigma,T)}
N_{\chi}(\sigma,T):=\#\{\rho=\beta+i\gamma\colon \beta\geq \sigma,~|\gamma|\leq T,~L(\rho,\chi)=0\}.
\end{equation}
It follows that $N(\sigma,Q)$ in \eqref{eqn:NchiDef} can be written as
\begin{equation}
\label{eqn:def_N(sigma,Q)}
N(\sigma,Q)=\sum_{q\leq Q}~\sum_{\substack{\chi\pmod{q} \\ \textup{$\chi$ primitive}}}N_{\chi}(\sigma,Q).
\end{equation}

\begin{lemma}
\label{lem:basic_density}
If $Q\geq 10^4$, then $N(0,Q)\leq 0.64 Q^3\log Q$.
\end{lemma}
\begin{proof}
If $\chi\pmod{q}$ is primitive, then by \cite[Corollary 1.2]{BennettMartinOBryant2} and  \cite[Theorem 19]{Rosser}, the quantity $N_{\chi}(\sigma,T)$ in \eqref{eqn:def_N_chi(sigma,T)} satisfies
\[
\Big|N_{\chi}(0,T)-\frac{T}{\pi}\log\frac{qT}{2\pi e}\Big|\leq\begin{cases}
\min\{0.247\log(qT)+6.894,0.298\log(qT)+4.358\}&\mbox{if $\chi\neq 1$ and $T\geq \frac{5}{7}$,}\\
0.137\log T+0.443\log\log T+2.463&\mbox{if $\chi=1$ and $T>14$.}
\end{cases}
\]
By \cref{lem:Platt}, if $0\leq T\leq 14$, then $N_1(0,T)=0$.  By \eqref{eqn:def_N(sigma,Q)}, these estimates imply the lemma.
\end{proof}

Our proof of \cref{thm:GLFZDE} relies crucially on our ability to efficiently count zeros in circles of small radius that are centered near the line $\re(s)=1$.

\begin{lemma}
	\label{lem:Linnik}
	Let $\sigma\geq 1$, $0<r\leq 1$, $q\geq 1$, $T\geq 1$, $t\in\R$, $|t|\leq T$, and $\chi\pmod{q}$ be a primitive Dirichlet character.  Define
	\begin{equation}
	\label{eqn:nchi_def}
	n_{\chi}(r,\sigma+it):=\#\{\rho=\beta+i\gamma\colon 0<\beta<1,~L(\rho,\chi)=0,~|\sigma+it-\rho|\leq r\}.
	\end{equation}
The bound $n_{\chi}(r,\sigma+it)\leq n_{\chi}(r,1+it)\leq n_{\chi}(2r,1+r+it)\leq r(2\log(qT)-1)+4$ holds.
\end{lemma}
\begin{proof}
	If $t\in\R$, $\sigma>1$, and $0<r\leq 1$, then it follows from the geometry of complex numbers that
	\begin{align*}
	n_{\chi}(r,\sigma+it)\leq n_{\chi}(r,1+it)\leq n_{\chi}(2r,1+r+it)&\leq 4r\sum_{|1+r+it-\rho|\leq 2r}\re\Big(\frac{1}{1+r+it-\rho}\Big)\\
	&\leq 4r\sum_{\rho}\re\Big(\frac{1}{1+r+it-\rho}\Big).
	\end{align*}
	By \eqref{eqn:Hadamard_real}, the last line equals
	\[
	4r\Big(\re\frac{L'}{L}(1+r+it,\chi)+\frac{1}{2}\log \frac{q}{\pi}+\frac{1}{2}\re\frac{\Gamma'}{\Gamma}\Big(\frac{1+r+\mathfrak{a}_{\chi}+it}{2}\Big)+\delta_{\chi}\re\Big(\frac{1}{1+r+it}+\frac{1}{r+it}\Big)\Big).
	\]
	Note that
	\[
	\re\Big(\frac{L'}{L}(1+r+it,\chi)\Big)\leq-\frac{\zeta'}{\zeta}(1+r)\leq\frac{1}{r},
	\]
	and \cref{lem:ono_sound} implies that (since $|t|\leq T$ and $T\geq 1$)
	\[
	\frac{1}{2}\re\frac{\Gamma'}{\Gamma}\Big(\frac{1+r+\mathfrak{a}_{\chi}+it}{2}\Big)\leq \frac{\frac{1}{2}\log((1+r+\mathfrak{a}_{\chi})^2+t^2)-\gamma_{\Q}}{2} \leq \frac{\log T+\frac{1}{2}\log 10-\gamma_{\Q}}{2}.
	\]
	Finally, if $\chi$ is trivial, then $L(s,\chi)=\zeta(s)$ does not vanish when $0\leq\re(s)\leq 1$ and $|\im(s)|\leq 14.13$.  Therefore, it suffices to assume that $|t|\geq 13.1$, in which case we have that
	\[
	\delta_{\chi}\re\Big(\frac{1}{1+r+it}+\frac{1}{r+it}\Big)\leq 0.017\,182\,3\delta_{\chi}.
	\]
	The lemma follows once we collect the estimates above.
\end{proof}

Following Heath-Brown \cite{HBLinnik}, we will substantially improve \cref{lem:Linnik} in a certain range of $r$; the quality of our improvement depends on how well one can explicitly bound $L(s,\chi)$ on the line $\re(s)=\frac{1}{2}$ in terms of $q$ and $|\im(s)|$.  We give a short and self-contained proof of an explicit version of the convexity bound $L(\frac{1}{2}+it,\chi)\ll (q(|t|+1))^{1/4}$ (without the usual $+\epsilon$ in the exponent).

\begin{proposition}
	\label{prop:sharp_convexity}
If $\chi\pmod{q_{\chi}}$ is primitive and $t\in\R$, then $|L(\tfrac{1}{2}+it,\chi)|\leq 2.97655(q_{\chi}|1+it|)^{\frac{1}{4}}$.
\end{proposition}
\begin{remark}
When $\chi\neq 1$, this improves \cite[(2) and (3)]{Hiary2} and \cite[Lemma 2.4]{Ramare} for all $q\geq 3$ and $t\in\R$.
\end{remark}

\begin{proof}
First, assume that $\chi\neq 1$.  By applying \cite[Lemma 3.1]{BennetSharpley} with $F(x) = L(x+it,\chi)$ evaluated at $x=\frac{1}{2}$, we have that
\[
\log|L(\tfrac{1}{2}+it,\chi)|\leq \frac{1}{2}\int_{\R}\log|L(i(t+v),\chi)L(1+i(t+v),\chi)|\frac{dv}{\cosh(\pi v)}.
\]
Since $L(z,\chi)=\overline{L(\bar{z},\bar{\chi})}$, it follows from \eqref{eqn:functional_equation} that
\[
|L(i(t+v),\chi)|=\sqrt{\frac{q_{\chi}}{\pi}}\Big|\frac{\Gamma(\frac{1+\mathfrak{a}_{\chi}+i(t+v)}{2})}{\Gamma(\frac{\mathfrak{a}_{\chi}-i(t+v)}{2})}L(1+i(t+v),\chi)\Big|.
\]
The bound
\[
\Big|\frac{\Gamma(\frac{1+\mathfrak{a}_{\chi}+i(t+v)}{2})}{\Gamma(\frac{\mathfrak{a}_{\chi}-i(t+v)}{2})}\Big|\leq \frac{|1+i(t+v)|^{\frac{1}{2}}}{\sqrt{2}}
\]
follows from the convexity bounds \cite[Lemma 1]{Rademacher} (when $\mathfrak{a}_{\chi}=0$, with $Q=\mathfrak{a}_{\chi}$ and $s=-(it+v)$) and \cite[Lemma 2]{Rademacher} (when $\mathfrak{a}_{\chi}=1$, with $Q=\mathfrak{a}_{\chi}$ and $s=-(it+v)$) for the $\Gamma$-function.  Both cases results in the bound
\[
|L(i(t+v),\chi)|\leq \sqrt{\frac{q_{\chi}}{2\pi}}|1+i(t+v)|^{\frac{1}{2}}|L(1+i(t+v),\chi)|.
\]
Since the growth of $\cosh(\pi v)$ in $v$ dominates the polynomial growth of $\log|L(1+\xi+i(t+v),\chi)|$ in $v$, the above integral is at most
\[
\lim_{\xi\to0^+}\int_{\R}\log|L(1+\xi+i(t+v),\chi)|\frac{dv}{\cosh(\pi v)}+\frac{1}{4}\int_{\R}\log|1+i(t+v)|\frac{dv}{\cosh(\pi v)}+\frac{1}{4}\log\frac{q _{\chi}}{2\pi}.
\]
By dominated convergence again, the limit equals
\begin{multline*}
\lim_{\xi\to0^+}\re\Big(\sum_{n=1}^{\infty}\frac{\Lambda(n)\chi(n)}{n^{1+\xi+it}\log n}\int_{\R}n^{-iv}\frac{dv}{\cosh(\pi v)}\Big)\\
=2\re\Big(\sum_{n=1}^{\infty}\frac{\Lambda(n)\chi(n)}{(n^{\frac{3}{2}+it}+n^{\frac{1}{2}+it})\log n}\Big)\leq 2\sum_{n=2}^{\infty}\frac{\Lambda(n)}{\log n}\sum_{j=0}^{\infty}\frac{(-1)^j}{n^{\frac{3}{2}+j}}= 2\sum_{j=0}^{\infty}(-1)^{j}\log\zeta(j+\tfrac{3}{2})\leq 1.4988.
\end{multline*}
If $t\in\R$, then
\[
\frac{1}{4}\int_{\R}\log|1+i(t+v)|\frac{dv}{\cosh(\pi v)}=\frac{1}{4}\log|1+it|+\frac{1}{8}\int_{\R}\log\Big|1+\frac{v^2+2|t|v}{t^2+1}\Big|\frac{dv}{\cosh(\pi v)}.
\]
If $w\geq -1$, we see via Taylor expansions that
\[
\lim_{w_0\to w^+}\log|1+w_0|\leq w-\frac{w^2}{2}+\frac{w^3}{3}.
\]
For all $t,v\in\R$, we have that $v^2+2|t|v=(v+|t|)^2-t^2\geq -t^2 >-(t^2+1)$, so
\[
\frac{1}{8}\int_{\R}\log\Big|1+\frac{v^2+2|t|v}{t^2+1}\Big|\frac{dv}{\cosh(\pi v)}\leq \frac{1}{8}\int_{\R}\sum_{j=1}^3\frac{(-1)^{j+1}}{j}\Big(\frac{v^2+2|t|v}{t^2+1}\Big)^j\frac{dv}{\cosh(\pi v)}=\frac{79+210t^2-48t^4}{1536(t^2+1)^3}.
\]
Since
\[
\sup_{t\in\R}\exp\Big(\frac{1}{4}\log\frac{1}{2\pi}+1.4988+\frac{79+210t^2-48t^4}{1536(t^2+1)^3}\Big)\leq 2.97655,
\]
the result follows.

If $\chi=1$, then $L(s,\chi)=\zeta(s)$.  In this case, the identity $\zeta(\bar{s})=\overline{\zeta(s)}$ and \cite[Theorem 1.1]{Hiary} yield
\[
|\zeta(\tfrac{1}{2}+it)|\leq \begin{cases}
	1.461&\mbox{if $0\leq |t|\leq 3$,}\\
	0.63 |t|^{1/6}\log |t|&\mbox{if $|t|\geq 3$.}
\end{cases}
\]
This finishes our proof in all cases.
\end{proof}

Following \cite{HBLinnik}, we combine \cref{prop:sharp_convexity} with a version of Jensen's formula to improve our bounds for the sum over zeros in \cref{lem:Linnik} in a more restrictive range of $r$ than in \cref{lem:Linnik}.

\begin{proposition}
\label{prop:HB_zero-count}
	Let $\chi\pmod{q}$ be primitive, $T\geq 1$, $t\in\R$, $|t|\leq T$, and $0<r<\frac{1}{2}$.  If $\chi=1$, then assume that $3\times 10^{12}\leq |t|\leq T$.  There holds
	\begin{align*}
	\sum_{\substack{L(\rho,\chi)=0 \\ |1+it-\rho|\leq r}}\re\Big(\frac{1}{1+r+it-\rho}\Big)&\leq \frac{\log(qT)+4.7098}{4+8r}+\frac{(\frac{4}{\pi}-1)\log(1+r^{-1})}{1+2r}+2.6908\\
	&+\frac{8r}{(1+2r)^2}(r(2\log(qT)-1)+4)+r^{-1}.
	\end{align*}
\end{proposition}

\begin{proof}
	Let $s=1+r+it$.  Choose $r_{1}\in(\max\{\frac{1}{2},2r\},\frac{1}{2}+r)$ so that no zeros of $L(z,\chi)$ lie on the circle $\{w\in\mathbb{C}\colon |s-w|=r_{1}\}$.  By \cite[Lemma 3.2]{HBLinnik} (with $f(z) = (\frac{z-1}{z+1})^{\delta_{\chi}}L(z,\chi)$ and $a=s$), we have that if
	\[
	J:=\int_0^{2\pi}\frac{\cos\theta}{\pi R}\log\Big|\Big(\frac{s+r_{1}e^{i\theta}-1}{s+r_{1}e^{i\theta}+1}\Big)^{\delta_{\chi}}L(s+r_{1}e^{i\theta},\chi)\Big|d\theta,
	\]
	then
\[
-\re\Big(\frac{L'}{L}(s,\chi)\Big)=\re\Big(\frac{2\delta_{\chi}}{s^2-1}\Big) - \sum_{|s-\rho|<r_1}\re\Big(\frac{1}{s-\rho}-\frac{s-\rho}{r_{1}^2}\Big)-J.
\]
Under our hypotheses, for $t$ and $T$ if $\chi=1$, we have the bound
\[
\delta_{\chi}\Big|\re\Big(\frac{2}{s^2-1}\Big)\Big|=\delta_{\chi}\Big|\frac{r}{r^2+t^2}-\frac{2+r}{(2+r)^2+t^2}\Big|\leq \frac{\delta_{\chi}}{10^{24}}.
\]
Since $0<r<2r<r_1$ and $0<\beta<1$, it follows using nonnegativity that
\begin{align*}
-\re\Big(\frac{L'}{L}(s,\chi)\Big)\leq \frac{\delta_{\chi}}{10^{25}}-\sum_{|1+it-\rho|\leq r}\re\Big(\frac{1}{s-\rho}\Big) + \sum_{|1+it-\rho|\leq r}\re\Big(\frac{s-\rho}{r_{1}^2}\Big)-J.
\end{align*}
If $|1+it-\rho|<r$, then $\re(s-\rho)=1-\beta+r<2r$.  Consequently, by \cref{lem:Linnik}, we have that
\[
\sum_{|1+it-\rho|\leq r}\re\Big(\frac{s-\rho}{r_1^2}\Big)\leq \frac{2r}{r_1^2}n_{\chi}(r,1+it)\leq \frac{2r}{r_1^2}(r(2\log(qT)-1)+4).
\]
It remains to bound $|J|$ and choose $r_{1}$.

We will work with the decomposition
\[
J=\int_0^{\frac{\pi}{2}-r}+\int_{\frac{\pi}{2}-r}^{\frac{\pi}{2}}+\int_{\frac{\pi}{2}}^{\frac{3\pi}{2}}+\int_{\frac{3\pi}{2}}^{\frac{3\pi}{2}+r}+\int_{\frac{3\pi}{2}+r}^{2\pi}=J_1+J_2+J_3+J_4+J_5.
\]
We observe that if $\theta\in[0,\frac{\pi}{2}]\cup[\frac{3\pi}{2},2\pi]$, then $\cos\theta\geq 0$ and
\[
\frac{\cos\theta}{\pi r_{1}}\log\Big|\Big(\frac{s+r_{1}e^{i\theta}-1}{s+r_{1}e^{i\theta}+1}\Big)^{\delta_{\chi}}L(s+r_{1}e^{i\theta},\chi)\Big|\leq \frac{\cos\theta}{\pi r_{1}}\log\zeta(1+r+r_{1}\cos\theta)\leq \frac{\cos\theta}{\pi r_{1}}\log\frac{1+r+r_{1}\cos\theta}{r+r_{1}\cos\theta}.	
\]
(We have used the fact that $\re(s+r_{1}e^{i\theta})\geq 1+r$, so $(\cos\theta)\log|\frac{s+r_{1}e^{i\theta}-1}{s+r_{1}e^{i\theta}+1}|\leq 0$ in our current range of $\theta$.)  Our hypotheses on $r$ and $r_{1}$ imply that if $0<\theta<\frac{\pi}{2}$, then
\[
\frac{1}{\pi r_{1}}\log\frac{1+r+r_{1}\cos\theta}{r+r_{1}\cos\theta}\leq \frac{2}{\pi}\log\frac{5}{2(r+r_{1}\cos\theta)}\leq \frac{2}{\pi}\log\Big(5\min\Big\{\frac{1}{\cos\theta},\frac{1}{2r}\Big\}\Big).
\]
Therefore,
\[
	J_1+J_2 \leq \frac{2}{\pi}\Big(\int_0^{\frac{\pi}{2}-r}\Big(\log\frac{5}{\cos\theta}\Big)\cos\theta d\theta+\Big(\log\frac{5}{2r}\Big)\int_{\frac{\pi}{2}-r}^{\frac{\pi}{2}}\cos\theta d\theta\Big)\leq 1.3454 .
\]
We similarly estimate $J_4+J_5$ using the symmetry of $\cos\theta$, concluding that
\[
J_1+J_2+J_4+J_5 \leq 2.6908 .
\]

We estimate $J_3$ with a different approach.  Let $\theta\in[\frac{\pi}{2},\frac{3\pi}{2}]$.  Since $2r<r_{1}<\frac{1}{2}+r$, we have that $\frac{1}{2}\leq 1+r+r_{1}\cos\theta\leq 1+r$.  Write $\sigma'=1+r+r_{1}\cos\theta$ and $y=t+r_{1}\sin\theta$ so that $s+r_{1}e^{i\theta}=\sigma'+iy$.  Also, if
\[
F(\sigma'+iy):=\Big(\frac{\sigma'+iy-1}{\sigma'+iy+1}\Big)^{\delta_{\chi}}L(\sigma'+iy,\chi),
\]
then
\[
F(\sigma'+iy)=\Big(\frac{s+r_1 e^{i\theta}-1}{s+r_1 e^{i\theta}+2}\Big)^{\delta_{\chi}}L(s+r_1 e^{i\theta},\chi).
\]
We estimate $|F(\sigma'+iy)|$ on the lines $\sigma'=\frac{1}{2}$  and $\sigma'=1+r$ using \cref{prop:sharp_convexity} and the bound
\[
|F(1+r+iu)|\leq \zeta(1+r^{-1})\leq 1+r^{-1},\qquad u\in\R,
\]
respectively.  We estimate $F(\sigma'+iy)$ in between these lines using the Phragm{\'e}n--Lindel{\"o}f principle; in particular,
\[
|F(\sigma'+iy)|\leq (e^{\frac{1}{4}\log(q|1+it|)+1.0908})^{\frac{1+r-\sigma'}{\frac{1}{2}+r}}(1+r^{-1})^{1-\frac{1+r-\sigma'}{\frac{1}{2}+r}}.
\]
(We have used the fact that if $\re(z)\geq \frac{1}{2}$, then $|\frac{z-1}{z+1}|\leq 1$.)  The ranges of $\theta$, $r$, and $r_{1}$ ensure that
\[
\log\Big|(e^{\frac{1}{4}\log(q|1+it|)+1.0908})^{-\frac{2r_{1} \cos\theta}{1+2r}}(1+r^{-1})^{\frac{1+2 r+2 r_{1} \cos\theta}{1+2r}}\Big|\geq 0,
\]
so
\[
|J_3|\leq \frac{1}{\pi r_{1}}\int_{\frac{\pi}{2}}^{\frac{3\pi}{2}}|\cos\theta|\cdot\log\Big|(e^{\frac{1}{4}\log(q|1+it|)+1.0908})^{-\frac{2r_{1} \cos\theta}{1+2r}}(1+r^{-1})^{\frac{1+2 r+2 r_{1} \cos\theta}{1+2r}}\Big|d\theta.
\]
Collecting our estimates for $J_1,\ldots,J_5$, we find that
\[
\lim_{r_1\to(r+\frac{1}{2})^{-}}|J|\leq \frac{\log(qT)+4.709\,774}{4+8r}+\frac{(\frac{4}{\pi}-1)\log(1+r^{-1})}{1+2r}+2.6908.
\]
Since
\[
\frac{4.709\,774}{4+8r}+\frac{\delta_{\chi}}{10^{24}}\leq \frac{4.7098}{4+8r},\qquad \re\Big(\frac{L'}{L}(1+r+it,\chi)\Big)\leq -\frac{\zeta'}{\zeta}(1+r)\leq\frac{1}{r}
\]
in our range of $r$, we arrive at the desired result.
\end{proof}

\begin{corollary}
\label{cor:Linnik_lemma}
	Let $Q\geq 3$, $T\geq 1$, and $\chi\pmod{q}$ be a primitive Dirichlet character with $q\leq Q$.  Let $t\in\R$ and $|t|\leq T$.  If $\chi=1$, then assume that $3\times 10^{12}\leq |t|\leq T$.  If $\chi\neq 1$, then assume that $\max\{Q,T\}>10\,000$.  If $n_{\chi}(r,s)$ is as in \eqref{eqn:nchi_def} and
	\[
	\frac{1}{3\log(QT)}\leq r\leq \frac{1}{10},
	\]
	then $n_{\chi}(r,1+it)\leq r(1+10^{-7})^{-1}(\frac{2}{3}\log(QT)+13.04)+2$.
\end{corollary}
\begin{remark}
The peculiar factor of $(1+10^{-7})^{-1}$ is for future convenience.	
\end{remark}

\begin{proof}
By our hypotheses on $Q$ and $T$, the range of $\eta$ is nonempty.  Let $s=1+r+it$.  If $\rho$ is a nontrivial zero of $L(s,\chi)$ and $|1+it-\rho|\leq r$, then
\[
\frac{1}{2r}\leq\re\Big(\frac{1}{1+r+it-\rho}\Big).
\]
Therefore, by \cref{prop:HB_zero-count} (which assumes that $3\times 10^{12}\leq|t|\leq T$ when $\chi=1$), we have that
\begin{multline*}
	n_{\chi}(r,1+it)\leq 2r\Big(\frac{\log(qT) +4.7908}{4+8r}+\frac{(\frac{4}{\pi}-1)\log(1+r^{-1})}{1+2r}+2.6908\\
	+\frac{8r}{(1+2r)^2}(r(2\log(qT)-1)+4)+r^{-1}\Big).
\end{multline*}
If $\frac{1}{3\log(QT)}\leq r\leq \frac{1}{10}$, then
\begin{align*}
&2\Big(\frac{\log(qT) +4.7908}{4+8r}+\frac{(\frac{4}{\pi}-1)\log(1+r^{-1})}{1+2r}+2.6908+\frac{8r}{(1+2r)^2}(r(2\log(qT)-1)+4)\Big)\\
&\leq \frac{23}{36}\log(QT)+\frac{5}{3}\Big(\frac{4}{\pi}-1\Big)\log(1+3\log(QT))+11.7111.
\end{align*}
The result now follows since $\max\{Q,T\}>10\,000$.
\end{proof}
\begin{remark}
Under the hypotheses of \cref{cor:Linnik_lemma}, our proof can be modified to produce an upper bound of the form $n_{\chi}(r,1+it)\leq \frac{1}{2}r\log(QT)+O(\log\log(QT))$ when $\frac{1}{3\log(eQT)}\leq r\leq \frac{1}{30}$.  But we cannot take full advantage of this due to considerations involving application of the ``pre-sifted'' large sieve inequality in \cref{lem:large_sieve} during our proof of \cref{thm:GLFZDE} in  \cref{sec:proof_main_theorem}.
\end{remark}

\subsection{The Deuring--Heilbronn phenomenon}

Here, we quantify the idea that if $\beta_1(Q)$ is close to $1$ in \cref{thm:GZFR}, then the zero-free region for all other Dirichlet $L$-functions (including $\zeta(s)$) improves as a specific function of how close $\beta_1(Q)$ is to $s=1$.  This was first quantified by Linnik \cite{Linnik} for low-lying zeros.  Our proof relies on the following lower bound for power sums.

\begin{lemma}[Theorem 2.2 of \cite{KNW}]
\label{lem:turan2}
	Let $\epsilon>0$.  Let $(z_n)_{n=1}^{\infty}$ (resp. $(b_n)_{n=1}^{\infty}$) be a sequence of complex (resp. nonnegative real) numbers with $z_1\neq 0$ (resp. $b_1>0$).  There exists an integer
	\[
	1\leq j \leq (8+\epsilon)\frac{1}{b_1}\sum_{n=1}^{\infty}\frac{b_n|z_n|}{|z_1|+|z_n|}
	\]
	such that
	\[
	\re\Big(\sum_{n=1}^{\infty}b_n z_n^j\Big)\geq \frac{\epsilon}{32+4\epsilon}b_1|z_1|^j.
	\]
\end{lemma}

\begin{corollary}
\label{cor:turan3}
	Let $\epsilon>0$.  Let $(z_n)_{n=1}^{\infty}$ be a sequence of complex numbers with $z_1\neq 0$.  Define $\mathcal{M}:=\frac{1}{|z_1|}\sum_{n=1}^{\infty}|z_n|$.  There exists an integer $1\leq j\leq (8+\epsilon)\mathcal{M}$ such that
	\[
	\re\Big(\sum_{n=1}^{\infty}z_n^j\Big)\geq \frac{\epsilon}{32+4\epsilon}|z_1|^j.
	\]
\end{corollary}
\begin{proof}
Since $\frac{|z_n|}{|z_1|+|z_n|}\leq \frac{|z_n|}{|z_1|}$ when  $z_1\neq 0$, this follows from \cref{lem:turan2} with $b_n=1$ for all $n\geq 1 $.
\end{proof}

We use \cref{cor:turan3} to prove what appears to be the first completely explicit version of the Deuring--Heilbronn zero repulsion phenomenon.

\begin{theorem}
\label{thm:DH}
Let $Q>400\,000$ and $T\geq 1$.  Recall \eqref{eqn:prod_L-functions} and \cref{thm:GZFR}.  Suppose that $\beta_1(Q)\geq 1-\frac{\Cr{ZFR}}{\log Q}$, and let $\chi_1\pmod{q_1}$ be the corresponding primitive quadratic character with $q_1\in(400\,000,Q]$.  If $\beta+i\gamma\neq \beta_1(Q)$ is a nontrivial zero of $\mathscr{L}(s,Q)$ with $\beta>\frac{1}{2}$ and $|\gamma|\leq T$, then
\[
\beta\leq 1-\frac{\log (\frac{1}{(1-\beta_1(Q))(670\,564.676 + 347\,029.502 \log Q + 107\,906.278\log T)})}{104.645 + 54.156 \log Q + 16.84\log T}.
\]
\end{theorem}

\begin{remark}
Our proof produces a stronger result than we have stated.  For numerical convenience, we assign specific values to the parameters $\epsilon$ and $\alpha$ occurring in our proof below.
\end{remark}

\begin{remark}
Jutila \cite[Theorem 2]{Jutila} proved a numerically stronger result under the assumption that $QT$ is sufficiently large (in terms of a parameter $\epsilon$ occurring in Jutila's proof but not ours).  To our knowledge, \cref{thm:DH} is the only explicit result of its type with no such hypotheses.
\end{remark}

\begin{proof}
The result is trivial when
\[
\beta_1(Q)\leq 1-\frac{1}{670\,564.676 + 347\,029.502 \log Q + 107\,906.278\log T},
\]
so we assume otherwise.  Let $1\leq q\leq Q$ be an integer, and let $\chi\pmod{q}$ be a primitive Dirichlet character.  Let $\beta+i\gamma\neq\beta_1(Q)$ be a nontrivial zero of $L(s,\chi)$ with $\beta>\frac{1}{2}$ and $|\gamma|\leq T$, and let $\psi\pmod{q_{\psi}}$ be the primitive character that induces $\chi_1\chi$.  Note that $q_{\psi}$ divides the least common multiple $[q,q_1]$.  Let $\omega$ denote a zero of the function
\[
D(s):=\zeta(s)L(s,\chi_1)L(s+i\gamma,\chi)L(s+i\gamma,\chi_1\chi)=\zeta(s)L(s,\chi_1)L(s+i\gamma,\chi)L(s+i\gamma,\psi)\prod_{p|[q,q_1]}\Big(1-\frac{\psi(p)}{p^{s+i\gamma}}\Big).
\]
There exists numbers $a_D,b_D\in\mathbb{C}$ such that
\[
(s-1)(s-1+i\gamma)^{\delta_{\chi}+\delta_{\psi}}D(s) = s^{\mathop{\mathrm{ord}}_{s=0}D(s)} e^{a_D+b_D s}\prod_{\substack{\omega\neq 0 \\ D(\omega) = 0}}\Big(1-\frac{s}{w}\Big)e^{\frac{s}{\omega}}.
\]
Let $\Lambda(n)$ be the von Mangoldt function, and define $a(n):=(1+\chi_1(n))(1+\chi(n)n^{-i\gamma})\Lambda(n)$.  By construction, we have the uniform lower bound $\re(a(n))\geq 0$.  Also note that if $\re(s)>1$, then
\[
-\frac{D'}{D}(s) = \sum_{n=1}^{\infty}\frac{a(n)\Lambda(n)}{n^s}.
\]

Let $\alpha\geq 1$.  For an integer $j\geq 1$, the $(2j-1)$-th derivative at $s=\alpha+1$ is
\[
\frac{1}{(2j-1)!}\sum_{n=1}^{\infty}\frac{a(n)\Lambda(n)(\log n)^{2j-1}}{n^{\alpha+1}}=\frac{1}{\alpha^{2j}}+\frac{\delta_{\chi}+\mathbf{1}_{\psi=1}}{(\alpha+i\gamma)^{2j}}-\sum_{\omega}\frac{1}{(\alpha+1-\omega)^{2j}}.
\]
We extract the contribution from $\beta_1(Q)$, arriving at
\begin{multline*}
\frac{1}{(2j-1)!}\sum_{n=1}^{\infty}\frac{a(n)\Lambda(n)(\log n)^{2j-1}}{n^{1+\alpha}}\\
=\frac{1}{\alpha^{2j}}+\frac{\delta_{\chi}+\delta_{\psi}}{(\alpha+i\gamma)^{2j}}-\frac{1}{(\alpha+1-\beta_1(Q))^{2j}}-\frac{\delta_{\chi}+\delta_{\psi}}{(\alpha+1-\beta_1(Q)+i\gamma)^{2j}}-\sum_{n=1}^{\infty}z_n^j,	
\end{multline*}
where $z_n = z_n(\gamma)$ runs through the (multi)sets
\begin{align*}
	&\{(1+\alpha-\omega)^{-2}\colon \zeta(\omega)=0\},\\
	&\{(1+\alpha-\omega)^{-2}\colon \omega\neq\beta_1(Q)\textup{ and }L(\omega,\chi_1)=0\},\\
	&\{(1+\alpha+i\gamma-\omega)^{-2}\colon \omega\neq\beta_1(Q)\textup{ and }L(\omega,\chi)=0\},\\
	&\{(1+\alpha+i\gamma-\omega)^{-2}\colon \omega\neq\beta_1(Q)\textup{ and }L(\omega,\psi)=0\},\\
	&\{(1+\alpha+i\gamma+(h+2k)\pi i/\log p)^{-2}\colon h\in\{0,1\},~k\in\Z,~p\nmid q_{\psi}\},
\end{align*}
arranged so that $z_1 = (\alpha+1+i\gamma-(\beta+i\gamma))^{-2}=(\alpha+1-\beta)^{-2}$.  Since $\re(a(n))\geq 0$, it follows that
\[
\re\Big(\frac{1}{(2j-1)!}\sum_{n=1}^{\infty}\frac{a(n)\Lambda(n)(\log n)^{2j-1}}{n^{1+\alpha}}\Big)\geq 0.
\]
Since $\delta_{\chi}+\mathbf{1}_{\psi=1}\in\{0,1\}$, it follows that
\begin{align*}
\re\Big(\sum_{n=1}^{\infty}z_n^j\Big)&\leq \frac{1}{\alpha^{2j}}-\frac{1}{(\alpha+1-\beta_1(Q))^{2j}}+(\delta_{\chi}+\mathbf{1}_{\psi=1})\re\Big(\frac{1}{(\alpha+i\gamma)^{2j}}-\frac{1}{(\alpha+1-\beta_1(Q)+i\gamma)^{2j}}\Big)\\
&=\frac{1}{\alpha^{2j}}-\frac{1}{(\alpha+1-\beta_1(Q))^{2j}}+(\delta_{\chi}+\mathbf{1}_{\psi=1})\re\Big(\frac{1}{(\alpha+i\gamma)^{2j}}\Big(1-\frac{1}{(1+\frac{1-\beta_1(Q)}{\alpha+i\gamma})^{2j}}\Big)\Big)\\
&\leq \frac{1}{\alpha^{2j}}\Big(1-\frac{1}{(1+\frac{1-\beta_1(Q)}{\alpha})^{2j}}+\Big|1-\frac{1}{(1+\frac{1-\beta_1(Q)}{\alpha+i\gamma})^{2j}}\Big|\Big).
\end{align*}
Since
\[
\Big|1-\frac{1}{(1+\frac{1-\beta_1(Q)}{\alpha+i\gamma})^{2j}}\Big|=\frac{2j}{|\alpha+i\gamma|}\Big|\int_{\beta_1(Q)}^1 \frac{1}{(1+\frac{1-u}{\alpha+i\gamma})^{2j+1}}du\Big|\leq \frac{2j(1-\beta_1(Q))}{|\alpha+i\gamma|}\leq \frac{2j(1-\beta_1(Q))}{\alpha},
\]
regardless of whether or not $\gamma=0$, we find that
\begin{equation}
	\label{eqn:DH_upper}
	\re\Big(\sum_{n=1}^{\infty}z_n^j\Big)\leq \frac{4j(1-\beta_1(Q))}{\alpha^{2j+1}}.
\end{equation}

We will apply \cref{cor:turan3} to bound the left hand side of \eqref{eqn:DH_upper} from below.  Since $\beta>\frac{1}{2}$, we have that $|z_1|\geq (\alpha+\frac{1}{2})^{-2}$, so $\mathcal{M}:=|z_1|^{-1}\sum_{n=1}^{\infty}|z_n|$ is at most
\begin{multline*}
\Big(\alpha+\frac{1}{2}\Big)^2\Big(\sum_{\zeta(\omega)=0}\frac{1}{|\alpha+1-\omega|^{2}}+\sum_{\substack{\omega\neq\beta_1(Q) \\ L(\omega,\chi_1)=0}}\frac{1}{|\alpha+1-\omega|^2}+\sum_{\substack{\omega\neq\beta_1(Q) \\ L(\omega,\chi)=0}}\frac{1}{|\alpha+1+i\gamma-\omega|^2}\\
+\sum_{\substack{\omega\neq\beta_1(Q) \\ L(\omega,\psi)=0}}\frac{1}{|\alpha+1+i\gamma-\omega|^2}+\sum_{\substack{p|[q,q_1] \\ p\nmid q_{\psi}}}\sum_{h=0}^{1}\sum_{k=-\infty}^{\infty}\frac{1}{|\alpha+1+i\gamma+\frac{(h+2k)\pi i}{\log p}|^2}\Big).
\end{multline*}
We apply \cite[Lemma 7.1]{TZ1} and its proof (with $K=\mathbb{Q}$) to bound the contribution not arising from the nontrivial zeros; this contribution is at most
\[
4\Big(\frac{1}{2(\alpha+1)}+\frac{1}{(\alpha+1)^2}\Big)+\Big(\frac{1}{2(\alpha+1)}+\frac{2}{(1+\alpha)^2\log 2}\Big)\log([q,q_1]).
\]
We apply \cite[Equation (7--2)]{TZ1} (with $K=\mathbb{Q}$, which ensures that $\psi$ is quadratic) to bound the contribution arising from the nontrivial zeros.  In their proof (with $K=\Q$), they have $Q=\max_{\chi}q_{\chi}$ so that $q,q_1\leq Q$ and $[q,q_1]\leq Q^2$.  These simplifying estimates are straightforward to reverse.  We conclude (using the bound $\delta_{\chi}+\mathbf{1}_{\chi_1\chi=1}\leq 1$) that
\begin{align*}
&\sum_{\zeta(\rho)=0}\frac{1}{|1+\alpha-\rho|^2}+\sum_{L(\rho,\chi_1)=0}\frac{1}{|\alpha+1-\rho|}+\sum_{L(\rho,\chi)=0}\frac{1}{|\alpha+i+i\gamma-\rho|}+\sum_{L(\rho,\psi)=0}\frac{1}{|1+\alpha+i\gamma-\rho|}
\\
&\leq \frac{1}{\alpha}\re\Big(\sum_{\zeta(\rho)=0}\frac{1}{1+\alpha-\rho}+\sum_{L(\rho,\chi_1)=0}\frac{1}{\alpha+1-\rho}+\sum_{L(\rho,\chi)=0}\frac{1}{\alpha+i+i\gamma-\rho}+\sum_{L(\rho,\psi)=0}\frac{1}{1+\alpha+i\gamma-\rho}\Big)\\
&\leq \frac{1}{\alpha}\Big(\frac{1}{2}\log(q_1 q [q,q_1])+\frac{\log([q,q_1])}{2^{\alpha+1}-1}+\log(\alpha+2+|\gamma|)+\log(\alpha+2)+\frac{3}{1+\alpha}+\frac{1}{2^{\alpha+1}-1}-2\log \pi\\
&+\frac{1}{\alpha}+\re\Big(\frac{\delta_{\chi}+\mathbf{1}_{\chi_1\chi=1}}{\alpha+it}+\frac{\delta_{\chi}+\mathbf{1}_{\chi_1\chi=1}}{\alpha+1+it}\Big)\Big)\\
&\leq \frac{1}{\alpha}\Big(\frac{1}{2}\log(q_1 q [q,q_1])+\frac{\log([q,q_1])}{2^{\alpha+1}-1}+\log T+2\log(\alpha+2)+\frac{6}{\alpha}+\log 2-2\log \pi\Big)
\end{align*}
When $\alpha\geq 12.8$, the above estimates (with $q,q_1\leq Q$ and $[q,q_1]\leq Q^2$) together imply that
\begin{equation}
\label{eqn:DH_M_bound}
\mathcal{M}\leq \frac{(\alpha+\frac{1}{2})^2}{\alpha}\Big(\frac{1}{2}\log(Q^4 T^2)+\Big(\frac{\alpha}{\alpha+1}+\frac{4\alpha}{(1+\alpha)^2\log 2}+\frac{2}{2^{\alpha+1}-1}\Big)\log Q+2\log(\alpha+2)+\frac{1}{\alpha}\Big).
\end{equation}

By \cref{cor:turan3} and a Taylor expansion argument, there exists $j_0\in[1,(8+\epsilon)\mathcal{M}]$ such that
\[
\re\Big(\sum_{n=1}^{\infty}z_n^{j_0}\Big)\geq \frac{\epsilon}{32+4\epsilon}|z_1|^{j_0}= \frac{\epsilon}{32+4\epsilon}\cdot\frac{1}{(\alpha+1-\beta)^{2j_0}}\geq \frac{\epsilon}{32+4\epsilon}\alpha^{-2j_0}e^{-\frac{2j_0}{\alpha}(1-\beta)}.
\]
We deduce from this and \eqref{eqn:DH_upper} the bound $\frac{\epsilon}{32+4\epsilon}e^{-2j_0(1-\beta)/\alpha}\leq 4j_0(1-\beta_1(Q))/\alpha$.  Since $j_0\in[1,(8+\epsilon)\mathcal{M}]$, we rearrange to solve for $\beta$, arriving at the bounds
\[
\beta\leq 1-\frac{\log(\frac{\epsilon}{(1-\beta_1(Q))16(8+\epsilon)^2\frac{\mathcal{M}}{\alpha}})}{2(8+\epsilon)\frac{\mathcal{M}}{\alpha}},\qquad 1\leq \frac{16(8+\epsilon)^2}{\epsilon}(1-\beta_1(Q))\frac{\mathcal{M}}{\alpha}e^{2(8+\epsilon)\frac{\mathcal{M}}{\alpha}(1-\beta)}.
\]
These two bounds imply the desired results once we use \eqref{eqn:DH_M_bound} with $\epsilon=\frac{1}{40}$ and $\alpha=23.22$ (and for the second part, we use $T=Q>400\,000$).
\end{proof}

\section{Notation, conventions, and the key optimization}
\label{sec:notation}

Throughout our proof of \cref{thm:GLFZDE}, we use the following notation and conventions that are crucial to our proofs:

\begin{enumerate}[(i)]
	\item $T\geq 1$, $Q\geq 1$, $1\leq q\leq Q$
	\item $\mathcal{L}=\log(QT)$,
	\item $\tau\in\R$ with $|\tau|\leq T$,
	\item $\alpha = 7.931\,643\,766\,252\,22\ldots$ (decimal expansion of a number whose origin is explained below),
	\item $A=3.907\,668\,327\,378\,39\ldots$  (decimal expansion of a number whose origin is explained below),
	\item $R=\sqrt{A^2+1}$,
	\item $\chi\pmod{q}$ a primitive Dirichlet character,
	\item $\frac{1}{3A\mathcal{L}}\leq \eta\leq \frac{1}{10A}$,
	\item $s_0=1+\eta+i\tau$,
	\item $\max\{Q,T\}>10^4$,
	\item $\delta_{\chi}$ the indicator function of the trivial primitive character,
	\item if $\delta_{\chi}=1$, then $3\times 10^{12}\leq |\tau|\leq T$,
	\item $\mathbf{1}_{I}(t)$ the indicator function of a subinterval $I\subseteq\R$,
	\item $\mathcal{N}_{\eta}=A\eta(1+10^{-7})^{-1}(\frac{2}{3}\mathcal{L}+13.04)+9$,
	\item $M_{\eta}=(\alpha-1)\mathcal{N}_{\eta}$, which is at least $63.9$
	\item $k\in[M_{\eta},M_{\eta}+\mathcal{N}_{\eta}-1]\cap\Z\subseteq [M_{\eta},\frac{\alpha}{\alpha-1}M_{\eta}]\cap\Z$, which implies that $k\geq 64$,
	\item $V=2(4e\alpha)^{1/(\alpha-1)}+0.38=4.184\,416\,849\ldots$,
	\item $A_0 = 1/(eV)=0.087\,916\,537\,56\ldots$,
	\item $A_{1} >2$ satisfies $V^{-1}=A_{1}  e^{1-\frac{A_{1} (\alpha-1)}{2\alpha}}$, so $A_{1}  = 11.065\,510\,190\ldots$,
	\item $N_{\eta} = \exp(A_{0} M_{\eta}/\eta)$,
	\item $N_{\eta}^{*}  = \exp(A_{1}  M_{\eta}/\eta)$,
	\item $\xi=1+10^{-7}$.
\end{enumerate}
\noindent When we recall one of these definitions or conventions, we will identify them as (i), (ii), (iii), etc.

The constants that arise in \cref{cor:Linnik_lemma} (see (xiv)) and the numbers $\alpha$ in (iv) and $A$ in (v) are central to these definitions and conventions.  Our values of $\alpha$ and $A$ arise as the values that
\begin{equation}
\label{eqn:minimization_problem}
\textup{minimize $(4e\alpha 2^{\alpha-1})^A$ with the constraints $\alpha > 1$, $A>1$, and $4e\alpha\Big(\frac{2}{\sqrt{A^2+1}}\Big)^{\alpha-1}=\frac{2}{3}$.}
\end{equation}
This optimization problem serves as the basis for the zero detection process in \cref{sec:zero_detection} (see \eqref{eqn:lowerbound1} below).  We could repeat our analysis with $\frac{2}{3}$ replaced by any fixed constant $0<\delta<1$, and the ensuing values of $\alpha$ and $A$ would depend on $\delta$.  As $\delta\to 1^-$, these new values of $\alpha$ and $A$ would lead to a versions of \cref{thm:GLFZDE} with worse implied constants and better exponents.  As $\delta\to 0^+$, the implied constants are better and the exponents are worse.  One cannot take $\delta$ to be too small because of considerations arising from our eventual use of the ``pre-sifted'' large sieve inequality in \cref{lem:large_sieve}.  Our choice of $\delta=\frac{2}{3}$ is a convenient middle ground between these competing interests that does not affect the quality of the power savings in \cref{thm:Sarkozy}.

The optimal choice of $\delta$ depends on the specific setting to which \cref{thm:GLFZDE} is applied.  For any fixed value $0<\delta<1$, our method will produce an explicit zero density estimate, so the interested reader can follow our ideas with the values of $\alpha$ and $A$ arising from a different value of $\delta$ and arrive at explicit versions of \eqref{eqn:Gallagher} and \eqref{eqn:Bombieri} with different constants.  We do not carry the dependence on $\delta$ throughout our proofs for clarity of exposition.

There is some flexibility in our choices.  We give two examples.  For reasons that will naturally arise in our proof, it is important that $V$ in (xvii) be greater than $2(4e\alpha)^{1/(\alpha-1)}$.  A large value of $V$ inflates the implied constants in \cref{thm:GLFZDE}, but a small value of $V$ that is too close to $2(4e\alpha)^{1/(\alpha-1)}$ will force us to needlessly inflate $\mathcal{N}_{\eta}$.  Our value of $V$ is convenient; there is no canonical choice.  All that matters for $\xi$ is that it is greater than 1, and the quality of the exponent (resp. the implied constant) in \cref{thm:GLFZDE} improves (resp. worsens) as $\xi\to 1^+$.

\section{The main zero detection result}
\label{sec:zero_detection}

The reader is strongly encouraged to frequently consult the notation and conventions in \cref{sec:notation}.  We proceed to detect the zeros of $L(s,\chi)$ that satisfy the following hypothesis.

\begin{hypothesis}
	\label{hyp}
Recall the notation and conventions in \cref{sec:notation}. There exists a nontrivial zero $\rho_0$ of $L(s,\chi)$ such that $|1+i\tau-\rho_0|\leq \eta$.
\end{hypothesis}

We believe that \cref{hyp} is false.  Indeed, it contradicts GRH.  Our main result en route to \cref{thm:GLFZDE} is the following zero detection result.

\begin{theorem}
\label{thm:detect}
Recall the notation and conventions in \cref{sec:notation}.  If \cref{hyp} is true, then
\[
1\leq \frac{\eta^3}{200}M_{\eta} e^{2.672\,32M_{\eta}}\int_{N_{\eta}}^{N_{\eta}^{*} }\Big|\sum_{N_{\eta}< p\leq u}\frac{\chi(p)\log p}{p^{1+i\tau}}\Big|^2 \frac{du}{u}.
\]
\end{theorem}
\begin{remark}
We remind the reader that if $\chi=1$, then one must have $3\times 10^{12}\leq|\tau|\leq T$ per (xii); otherwise, \cref{hyp} is false by \cref{lem:RH}.  If $\chi\neq 1$, then one must have $\max\{q,T\}>10^4$ per (x); otherwise, \cref{hyp} is false by \cref{lem:Platt}.
\end{remark}

Our proof relies on another lower bound for power sums, different from \cref{lem:turan2}.  The result that we cite, due to Kolesnik and Strauss, descends from original arguments of S{\'o}s and Tur{\'a}n \cite{Turan}.

\begin{lemma}[\cite{KolesnikStraus}]
	\label{lem:turan}
	Let $M\geq 0$ and $N\geq 1$ be integers, and let $z_1,\ldots,z_N\in\Z$.  There exists an integer $k\in[M+1,M+N]$ such that
	\[\
	|z_1^k+\cdots+z_N^k|\geq 1.007\Big(\frac{1}{4e(1+\frac{M}{N})}\Big)^N |z_1|^k.
	\]
\end{lemma}

\begin{remark}
The constant $4e$ is essentially optimal; see \cite{Makai}.
\end{remark}

\subsection{A conditional lower bound on high derivatives}

Assume \cref{hyp}. By \eqref{eqn:Hadamard}, if $\re(s)>1$, then for $k$ satisfying (xvi), we have that
\[
\frac{(-1)^{k+1}}{k!}\Big(\frac{L'}{L}(s,\chi)\Big)^{(k)}=\delta_{\chi}\Big(\frac{1}{s^{k+1}}+\frac{1}{(s-1)^{k+1}}\Big)-\sum_{\rho}\frac{1}{(s-\rho)^{k+1}}-\sum_{n=0}^{\infty}\frac{1}{(s+2(n+\delta_{\chi})+\mathfrak{a}_{\chi})^{k+1}}.
\]
We evaluate the above expression at $s=s_0$, with $s_0$ as in (ix).  Some casework shows that
\[
\Big|\frac{\delta_{\chi}}{s_0^{k+1}}-\sum_{n=0}^{\infty}\frac{1}{(s_0+2(n+\delta_{\chi})+\mathfrak{a}_{\chi})^{k+1}}\Big|\leq (1-2^{-k-1})\zeta(k+1)\leq 1+k^{-2}.
\]
By (xii), we have that
\[
\delta_{\chi}\Big|\frac{1}{(s_0-1)^{k+1}}\Big|\leq \frac{\delta_{\chi}}{10^{12(k+1)}}.
\]
If $|1+i\tau-\rho|>A\eta$, then $|s_0-\rho|\geq R\eta$.  Therefore, by an intermediate calculation in the proof of \cref{lem:Linnik}, we have that
\begin{align*}
\sum_{|1+i\tau-\rho|>A\eta}\frac{1}{|s_0-\rho|^{k+1}}&\leq \frac{1}{(R\eta)^{k-1}}\sum_{\rho}\frac{1}{|s_0-\rho|^2}\\
&\leq \frac{1}{(R\eta)^{k-1}\eta}\sum_{\rho}\frac{1+\eta-\beta}{|s_0-\rho|^2}=\frac{R}{(R\eta)^{k}}\sum_{\rho}\re\Big(\frac{1}{s_0-\rho}\Big)\leq \frac{(\eta\mathcal{L}+2)R^2}{2(R\eta)^{k+1}}.
\end{align*}
Combining the above results in this subsection and applying the lower bound on $k$ in (xvi), we arrive at the bound
\begin{equation}
\label{eqn:prelim_lower_bound_zeros}
\begin{aligned}
&\frac{\eta^{k+1}}{k!}\Big|\Big(\frac{L'}{L}\Big)^{(k)}(s_0,\chi)\Big|\\
&\geq \Big|\sum_{|1+i\tau-\rho|\leq A\eta}\frac{\eta^{k+1}}{(s_0-\rho)^{k+1}}\Big|-\frac{1}{R^{k+1}}\Big(\frac{\eta\mathcal{L}+2}{2}R^2+(1+k^{-2}+10^{-12(k+1)})(R\eta)^{k+1}\Big)\\
&\geq \Big|\sum_{|1+i\tau-\rho|\leq A\eta}\frac{\eta^{k+1}}{(s_0-\rho)^{k+1}}\Big|-\frac{1}{R^{k+1}}\Big(\frac{\eta\mathcal{L}+2}{2}R^2+1.001(R\eta)^{k+1}\Big).
\end{aligned}
\end{equation}

By \cref{hyp}, the sum over zeros in \eqref{eqn:prelim_lower_bound_zeros} is nonempty, which makes \cref{lem:turan} applicable.  In the notation of \cref{lem:turan}, we have that
\[
N = n_{\chi}(A\eta,1+i\tau)\geq 1,\qquad \{z_1,\ldots,z_N\}=\{(s_0-\rho)^{-1}\colon |1+i\tau-\rho|\leq A\eta\},\qquad z_1 = \rho_0.
\]
Note that $N \leq\mathcal{N}_{\eta}$ (see (xiv)) per \cref{cor:Linnik_lemma}.  With $M_{\eta}$ as in (xv), we choose
\[
M=M_{\eta}.
\]
Per \cref{hyp}, we have that $|s_0-\rho_0|=|(1+i\tau-\rho_0)+\eta|\leq 2\eta$.  It follows that there exists $k\in[M,M+N-1]$, necessarily greater than or equal to $64$ per (xvi), such that
\[
\Big|\sum_{|1+i\tau-\rho|\leq A\eta}\frac{\eta^{k+1}}{(s_0-\rho)^{k+1}}\Big|\geq 1.007\Big(\frac{1}{4e(1+\frac{M}{N})}\Big)^{N}\frac{\eta^{k+1}}{|s_0-\rho_0|^{k+1}}\geq \frac{1.007}{2^{k+1}}\Big(\frac{1}{4e(1+\frac{M}{N})}\Big)^{N}.
\]
This lower bound and  \eqref{eqn:prelim_lower_bound_zeros} imply that
\begin{equation}
	\label{eqn:prelim_lower_bound_zeros2}
	\frac{\eta^{k+1}}{k!}\Big|\Big(\frac{L'}{L}\Big)^{(k)}(s_0,\chi)\Big|\geq \frac{1.007}{2^{k+1}}\Big(\frac{1}{4e(1+\frac{M}{N})}\Big)^{N}-\frac{1}{R^{k+1}}\Big(\frac{\eta\mathcal{L}+2}{2}R^2+1.001(R\eta)^{k+1}\Big).
\end{equation}

Using the bound $k+1\geq M_{\eta}=(\alpha-1)\mathcal{N}_{\eta}$ as well as (vi), (x), (xiv), (xv), and (xvi),  we find that \eqref{eqn:prelim_lower_bound_zeros2} is at least
\begin{equation}
\begin{aligned}
\label{eqn:lowerbound1}
&1.007\Big(\frac{1}{4e\alpha}\Big)^{\mathcal{N}_{\eta}}\frac{1}{2^{k+1}}-\frac{1}{R^{k+1}}\Big(\frac{\eta\mathcal{L}+2}{2}R^2+1.001(R\eta)^{k+1}\Big)\\
&\geq 1.007\Big(\frac{1}{4e\alpha}\Big)^{\mathcal{N}_{\eta}}\frac{1}{2^{k+1}}\Big(1-(4e\alpha)^{\mathcal{N}_{\eta}}\frac{2^{k+1}}{R^{k+1}}3.123\mathcal{N}_{\eta}\Big)\\
&\geq 1.007\Big(\frac{1}{(4e\alpha 2^{\alpha-1})^A}\Big)^{\frac{\mathcal{N}_{\eta}}{A}}\Big(1-\Big(4e\alpha\Big(\frac{2}{\sqrt{A^2+1}}\Big)^{\alpha-1}\Big)^{\mathcal{N}_{\eta}}3.123\mathcal{N}_{\eta}\Big).
\end{aligned}
\end{equation}
It is at this point that the minimization problem \eqref{eqn:minimization_problem} arises, and our choices of $\alpha$ in (iv) and $A$ in (v) are justified.  (Notice that the coefficient of $\eta\mathcal{L}$ in $\mathcal{N}_{\eta}/A$ is independent of $A$.)  Now, \eqref{eqn:lowerbound1} equals
\begin{equation}
\begin{aligned}
\label{eqn:lowerbound2}
&1.007\Big(\frac{1}{(4e\alpha 2^{\alpha-1})^{\frac{1}{\alpha-1}}}\Big)^{M_{\eta}}\Big(1-\Big(\frac{2}{3}\Big)^{\mathcal{N}_{\eta}}3.123\mathcal{N}_{\eta}\Big)\\
&=1.007\Big(\frac{1}{3.804\,416\,849\,672\ldots}\Big)^{M_{\eta}}\Big(1-\Big(\frac{2}{3}\Big)^{\mathcal{N}_{\eta}}3.123\mathcal{N}_{\eta}\Big).
\end{aligned}
\end{equation}
Since $\mathcal{N}_{\eta}\geq 9$ per (xiv), it follows that
\begin{equation}
	\label{eqn:main_lower_bound}
\frac{\eta^{k+1}}{k!}\Big|\Big(\frac{L'}{L}\Big)^{(k)}(s_0,\chi)\Big|\geq 0.27\Big(\frac{1}{3.804\,416\,849\,673}\Big)^{M_{\eta}}.
\end{equation}

\subsection{Unconditional upper bounds for high derivatives}

Now, we turn to an upper bound for
\begin{equation}
\label{eqn:deriv_expansion}
\frac{\eta^{k+1}}{k!}\Big|\Big(\frac{L'}{L}\Big)^{(k)}(s_0,\chi)\Big| = \eta\Big|\sum_{n=1}^{\infty}\frac{\Lambda(n)\chi(n)}{n^{1+i\tau}}j_{k}(\eta\log n)\Big|,\qquad j_{k}(u) = e^{-u}\frac{u^k}{k!}.	
\end{equation}
Per (xvi), we have that $k\in[(\alpha-1)\mathcal{N}_{\eta},\alpha\mathcal{N}_{\eta}]\subseteq[M_{\eta},\frac{\alpha}{\alpha-1}M_{\eta}]$ and $k\geq 64$.  Recall the definitions of $V$ (xvii), $A_0$ (xviii), $A_1$ (xix), $N_{\eta}$ (xx), and $N_{\eta}^*$ (xxi).  We decompose the right hand side of \eqref{eqn:deriv_expansion} into a sum over $n\notin(N_{\eta},N_{\eta}^{*} ]$, a sum over composite $n\in(N_{\eta},N_{\eta}^{*} ]$, and a (dominant) sum over prime $n\in(N_{\eta},N_{\eta}^{*} ]$.

If $n\leq N_{\eta}$, then $\eta\log n\leq A_{0} M_{\eta}$.  Since $k\in[M_{\eta},\frac{\alpha}{\alpha-1}M_{\eta}]$ and $k!\geq(k/e)^k$, it follows that
\[
j_{k}(\eta\log n)=n^{-\eta}\frac{(\eta\log n)^k}{k!}\leq n^{-\eta}\Big(\frac{A_{0} eM_{\eta}}{k}\Big)^k\leq n^{-\eta}V^{-k}.
\]
If $n\geq N_{\eta}^{*} $, then $\eta\log n\geq A_{1}  M_{\eta}$.  Since $e^{-u/2}u^k/k!$ is decreasing for $u>2k$ and $A_{1} >2$, we find that if $n\geq N_{\eta}^{*} $, then
\begin{align*}
j_{k}(\eta\log n)=n^{-\frac{\eta}{2}}\frac{e^{-\frac{1}{2}\eta\log n}(\eta\log n)^{k}}{k!}\leq n^{-\frac{\eta}{2}}\frac{e^{-\frac{1}{2}A_{1}  M_{\eta}}(A_{1}  M_{\eta})^k}{k!}&\leq n^{-\frac{\eta}{2}}e^{-\frac{1}{2}A_{1}  M_{\eta}}(A_{1}  M_{\eta})^k\Big(\frac{e}{k}\Big)^k\\
&= n^{-\frac{\eta}{2}}(A_{1}  e^{1-\frac{1}{2}A_{1}  \frac{M_{\eta}}{k}})^k \Big(\frac{M_{\eta}}{k}\Big)^k\\
&\leq n^{-\frac{\eta}{2}}(A_{1}  e^{1-\frac{\alpha-1}{2\alpha}A_{1} })^k= n^{-\frac{\eta}{2}}V^{-k}.
\end{align*}
Therefore, we have the bound
\[
\eta\Big|\sum_{n\notin(N_{\eta},N_{\eta}^{*} ]}\frac{\Lambda(n)\chi(n)}{n^{1+i\tau}}j_{k}(\eta\log n)\Big|\leq \frac{\eta}{V^k}\Big(\sum_{n=1}^{\infty}\frac{\Lambda(n)}{n^{1+\eta}}+\sum_{n=1}^{\infty}\frac{\Lambda(n)}{n^{1+\frac{\eta}{2}}}\Big)\leq \frac{\eta}{V^k}\Big(\frac{1}{\eta}+\frac{2}{\eta}\Big)=\frac{3}{V^k}.
\]

Now, let $ n\in(N_{\eta},N_{\eta}^{*} ]$.  Since $(\log u)^k\leq k!u$ for $k,u\geq 1$ and $0<\eta\leq\frac{1}{10A}$ per (viii), we have that
\[
j_{k}(\eta\log n)=n^{-\eta}\frac{(\eta\log n)^k}{k!}=n^{-\eta}(2\eta)^k\frac{(\log n^{\frac{1}{2}})^k}{k!}\leq (2\eta)^k n^{\frac{1}{2}-\eta}\leq\frac{n^{\frac{1}{2}-\eta}}{(3V)^k}.
\]
Thus,
\begin{align*}
\eta\Big|\sum_{\substack{ n\in(N_{\eta},N_{\eta}^{*} ] \\ \textup{$n$ composite}}}\frac{\Lambda(n)\chi(n)}{n^{1+i\tau}}j_{k}(\eta\log n)\Big|&\leq \frac{\eta}{(3V)^k}\sum_{\substack{ n\in(N_{\eta},N_{\eta}^{*} ] \\ \textup{$n$ composite}}}\frac{\Lambda(n)}{n^{\frac{1}{2}+\eta}}\\
&=\frac{\eta}{(3V)^k}\sum_{p}\sum_{j=2}^{\infty}\frac{\log p}{p^{(\frac{1}{2}+\eta)j}}\\
&\leq (2+\sqrt{2})\frac{\eta}{(3V)^k}\sum_{p}\frac{\log p}{p^{1+2\eta}}\\
&\leq -(2+\sqrt{2})\frac{\eta}{(3V)^k}\frac{\zeta'}{\zeta}(1+2\eta)\leq (2+\sqrt{2})\frac{\eta}{(3V)^k}\frac{1}{2\eta}=\frac{1+2^{-\frac{1}{2}}}{(3V)^k}.
\end{align*}

Finally, we use partial summation to consider the sum
\begin{multline*}
\eta\Big|\sum_{p \in(N_{\eta},N_{\eta}^{*} ]}\frac{\Lambda(p)\chi(p)}{p^{1+i\tau}}j_{k}(\eta\log p)\Big|=\eta\Big|j_{k}(\eta\log N_{\eta}^{*} )\sum_{p \in(N_{\eta},N_{\eta}^{*} ]}\frac{\chi(p)\log p}{p^{1+i\tau}}\\
-\int_{N_{\eta}}^{N_{\eta}^{*} }\sum_{p \in(N_{\eta},u]}\frac{\chi(p)\log p}{p^{1+i\tau}}\Big(\frac{d}{du}j_{k}(\eta\log u)\Big)du\Big|.	
\end{multline*}
Our earlier bound for $j_k(\eta\log n)$ when $n\geq N_{\eta}^{*}$ implies that
\[
\eta\Big|j_{k}(\eta\log N_{\eta}^{*} )\sum_{p \in(N_{\eta},N_{\eta}^{*} ]}\frac{\chi(p)\log p}{p^{1+i\tau}}\Big|\leq \frac{\eta(N_{\eta}^{*})^{-\frac{\eta}{2}}}{V^{k}}\sum_{p \leq N_{\eta}^{*}}\frac{\log p}{p}\leq \frac{\eta}{V^{k}}\sum_{p \leq N_{\eta}^{*}}\frac{\log p}{p^{1+\frac{\eta}{2}}}\leq\eta \frac{2e}{\eta V^{k}} \leq \frac{2e}{V^k}.
\]
Since $k\geq 64$, it follows that
\[
\Big|\frac{d}{du}j_{k}(\eta\log u)\Big| = \frac{\eta}{u}|(j_{k-1}(\eta\log u)-j_{k}(\eta\log u))|\leq \frac{\eta}{u}\sup_{\substack{u\geq 1 \\ k\geq 64 \\ \eta>0}}\frac{u^{-\eta}(\eta\log u)^{k-1}|k-\eta\log u|}{k!} \leq \frac{\eta}{200u}.
\]

We combine the results in this subsection to conclude that
\begin{equation}
\frac{\eta^{k+1}}{k!}\Big|\Big(\frac{L'}{L}\Big)^{(k)}(s_0,\chi)\Big|\leq \frac{\eta^2}{200}\int_{N_{\eta}}^{N_{\eta}^{*} }\Big|\sum_{N_{\eta}< p\leq u}\frac{\chi(p)\log p}{p^{1+i\tau}}\Big|\frac{du}{u}+\frac{8.437}{V^k}.
\end{equation}
Therefore,
\begin{equation}
	\label{eqn:main_upper_bound}
\frac{\eta^{k+1}}{k!}\Big|\Big(\frac{L'}{L}\Big)^{(k)}(s_0,\chi)\Big|\leq \frac{\eta^2}{200}\int_{N_{\eta}}^{N_{\eta}^{*}}\Big|\sum_{N_{\eta}< p\leq u}\frac{\chi(p)\log p}{p^{1+i\tau}}\Big|\frac{du}{u}+\frac{8.437}{V^k}.
\end{equation}

\subsection{Proof of \cref{thm:detect}}

Under \cref{hyp}, the lower bound \eqref{eqn:main_lower_bound} holds.  Unconditionally, the upper bound \eqref{eqn:main_upper_bound} holds.  Therefore, under \cref{hyp}, the two bounds imply that
\[
0.27\Big(\frac{1}{3.804\,416\,849\,673}\Big)^{M_{\eta}}\leq \frac{\eta^2}{200}\int_{N_{\eta}}^{N_{\eta}^{*} }\Big|\sum_{N_{\eta}< p\leq u}\frac{\chi(p)\log p}{p^{1+i\tau}}\Big|\frac{du}{u}+\frac{8.437}{V^k}.
\]
It follows from our lower bounds for $k$ and $M_{\eta}$ in (xv) and (xvi) as well as our value of $V$ in (xvii) that
\[
\frac{1}{4}e^{-1.336\,162\,721M_{\eta}}\leq \frac{\eta^2}{200}\int_{N_{\eta}}^{N_{\eta}^{*} }\Big|\sum_{N_{\eta}< p\leq u}\frac{\chi(p)\log p}{p^{1+i\tau}}\Big|\frac{du}{u}.
\]
We square both sides and apply the Cauchy--Schwarz inequality to obtain
\begin{align*}
\frac{1}{16}e^{-2.672\,32M_{\eta}}&\leq \Big(\frac{\eta^2}{200}\int_{N_{\eta}}^{N_{\eta}^{*} }\Big|\sum_{N_{\eta}< p\leq u}\frac{\chi(p)\log p}{p^{1+i\tau}}\Big|\frac{du}{u}\Big)^2\\
&\leq \frac{\eta^4}{40\,000}\Big(\int_{N_{\eta}}^{N_{\eta}^{*} }\frac{du}{u}\Big)\int_{N_{\eta}}^{N_{\eta}^{*} }\Big|\sum_{N_{\eta}< p\leq u}\frac{\chi(p)\log p}{p^{1+i\tau}}\Big|^2 \frac{du}{u}\\
&\leq \frac{11}{40\,000}\eta^3 M_{\eta}\int_{N_{\eta}}^{N_{\eta}^{*} }\Big|\sum_{N_{\eta}< p\leq u}\frac{\chi(p)\log p}{p^{1+i\tau}}\Big|^2 \frac{du}{u}.
\end{align*}
\cref{thm:detect} follows once we multiply both sides by $16e^{2.672\,32M_{\eta}}$.

\subsection{From detecting zeros to counting zeros}

We apply \cref{thm:detect} to count zeros of $L(s,\chi)$ that satisfy  $|1+i\tau-\rho|\leq \eta$.  In order to account for zeros that lie on the boundary, we increase the radius by a factor of $\xi$ (see (xxii)).  Therefore, in all that follows, every instance of $\eta$ from before must be increased by a factor of $\xi$.

Let $\rho=\beta+i\gamma$ be a nontrivial zero of $L(s,\chi)$.  For $w>1$, define
\[
\Upphi_{\rho,w}(\tau)=\Upphi_{\rho,w,\chi}(\tau)=\begin{cases}
	1&\mbox{if $|1+i\tau-\rho|\leq w$,}\\
	0&\mbox{otherwise.}
\end{cases}
\]
Regardless of whether $\rho_0$ satisfies \cref{hyp}, it follows that
\[
\Upphi_{\rho_0,\xi\eta}(\tau)\leq \Big(\frac{1}{200}(\xi\eta)^3 M_{\xi\eta} e^{2.672\,32M_{\xi\eta}}\int_{N_{\xi\eta}}^{N_{\xi\eta}^{*} }\Big|\sum_{N_{\xi\eta}< p\leq u}\frac{\chi(p)\log p}{p^{1+i\tau}}\Big|^2 \frac{du}{u}\Big)\Upphi_{\rho_0,\xi\eta}(\tau).
\]
We integrate both sides over $\tau\in[-T,T]$ and obtain
\[
\int_{-T}^{T}\Upphi_{\rho_0,\xi\eta}(\tau)d\tau\leq \frac{1}{200}(\xi\eta)^3 M_{\xi\eta} e^{2.672\,32M_{\xi\eta}}\int_{-T}^{T}\Big(\int_{N_{\xi\eta}}^{N_{\xi\eta}^{*} }\Big|\sum_{N_{\xi\eta}< p\leq u}\frac{\chi(p)\log p}{p^{1+i\tau}}\Big|^2 \frac{du}{u}\Big)\Upphi_{\rho_0,\xi\eta}(\tau)d\tau.
\]
Because of \cref{lem:RH}, this maneuver is acceptable even with (xii).  If $1-\eta\leq\re(\rho_0)<1$ and $|\im(\rho_0)|\leq T$, then under \cref{hyp}, we have that
\[
\int_{-T}^{T}\Upphi_{\rho_0,\xi\eta}(\tau)d\tau \geq \eta\sqrt{\xi^2-1}.
\]
Consequently, if $1-\eta\leq\re(\rho_0)<1$ and $|\im(\rho_0)|\leq T$, then by \cref{thm:detect},
\begin{equation}
\label{eqn:zero_detect_1}
1\leq \frac{\xi^3}{200\sqrt{\xi^2-1}}\eta^2 M_{\xi\eta} e^{2.672\,32M_{\xi\eta}}\int_{-T}^{T}\Big(\int_{N_{\xi\eta}}^{N_{\xi\eta}^{*} }\Big|\sum_{N_{\xi\eta}< p\leq u}\frac{\chi(p)\log p}{p^{1+i\tau}}\Big|^2 \frac{du}{u}\Big)\Upphi_{\rho_0,\xi\eta}(\tau)d\tau.
\end{equation}
\cref{cor:Linnik_lemma}, (xiv), and (xv) imply that
\[
\sum_{\substack{\rho \\ L(\rho,\chi)=0}}\Upphi_{\rho}(\tau)=n_{\chi}(\xi\eta,1+i\tau)\leq \frac{\xi}{\alpha-1}M_{\xi\eta}.
\]
Therefore, upon summing \eqref{eqn:zero_detect_1} over all nontrivial zeros $\rho=\beta+i\gamma$ of $L(s,\chi)$ with $1-\eta\leq\beta\leq 1$ and $|\gamma|\leq T$, we find that if
\[
N_{\chi}(\sigma,T):=\#\{\rho=\beta+i\gamma\colon \beta>\sigma,~|\gamma|\leq T,~L(\rho,\chi)=0\},
\]
then
\begin{equation}
	\label{eqn:ZDE_dirichlet_polynomial1}
\begin{aligned}
&N_{\chi}(1-\eta,T)\\
&\leq \frac{\xi^3}{200\sqrt{\xi^2-1}}\eta^2 M_{\xi\eta} e^{2.672\,32M_{\xi\eta}}\int_{-T}^{T}\Big(\int_{N_{\xi\eta}}^{N_{\xi\eta}^{*} }\Big|\sum_{N_{\xi\eta}< p\leq u}\frac{\chi(p)\log p}{p^{1+i\tau}}\Big|^2 \frac{du}{u}\Big)\sum_{\rho}\Upphi_{\rho_0,\xi\eta}(\tau)d\tau\\
&\leq  1.613\eta^2 M_{\xi\eta}^2 e^{2.672\,32M_{\xi\eta}}\int_{-T}^{T}\Big(\int_{N_{\xi\eta}}^{N_{\xi\eta}^{*} }\Big|\sum_{N_{\xi\eta}< p\leq u}\frac{\chi(p)\log p}{p^{1+i\tau}}\Big|^2 \frac{du}{u}\Big)d\tau\\
&=1.613\eta^2 M_{\xi\eta}^2 e^{2.672\,32M_{\xi\eta}}\int_{N_{\xi\eta}}^{N_{\xi\eta}^{*} }\Big(\int_{-T}^{T}\Big|\sum_{N_{\xi\eta}< p\leq u}\frac{\chi(p)\log p}{p^{1+i\tau}}\Big|^2 d\tau\Big)\frac{du}{u}.
\end{aligned}
\end{equation}

\section{Proof of \cref{thm:GLFZDE}}
\label{sec:proof_main_theorem}

We begin by summing \eqref{eqn:ZDE_dirichlet_polynomial1} over all primitive Dirichlet characters with modulus at most $Q$.  In keeping with \cite{Bombieri2,Gallagher}, we will choose $T = Q$, so (x) becomes the condition $Q>10^4$.  We record
\begin{equation}
\label{eqn:Nchoice_large_sieve}
\begin{aligned}
\mathcal{N}_{\xi\eta} &= A\eta\Big(\frac{4}{3}\log Q+13.04\Big),\\
M_{\xi\eta}&=(\alpha-1)\mathcal{N}_{\eta}\geq 63.925,\\
N_{\xi\eta}&=\exp\Big(\frac{A_0 M_{\xi\eta}}{\xi\eta}\Big)>10^{106}Q^{3.175\,142}.
\end{aligned}
\end{equation}

\subsection{Proof when $Q\leq 10^4$}

Since $\sigma\geq\frac{39}{40}$ in \cref{thm:GLFZDE}, this follows immediately from \cref{lem:Platt}.

\subsection{Bound for $N(\sigma,Q)$ when $Q>10^4$}
\label{subsec:nonexc}

We sum \eqref{eqn:ZDE_dirichlet_polynomial1} over all primitive $\chi\pmod{q_{\chi}}$ with $q_{\chi}\leq Q$, leading us to now bound
\begin{equation}
\label{eqn:averaged_N}
\begin{aligned}
&\sum_{q\leq Q}\sum_{\substack{\chi\pmod{q} \\ \textup{$\chi$ primitive}}}N_{\chi}(1-\eta,T)\\
&\leq 1.613\eta^2 M_{\xi\eta}^2 e^{2.672\,32M_{\xi\eta}}\int_{N_{\xi\eta}}^{N_{\xi\eta}^{*} }\Big(\sum_{q\leq Q}\sum_{\substack{\chi\pmod{q} \\ \textup{$\chi$ primitive}}}\int_{-T}^{T}\Big|\sum_{N_{\xi\eta}< p\leq u}\frac{\chi(p)\log p}{p^{1+i\tau}}\Big|^2 d\tau\Big)\frac{du}{u}.	
\end{aligned}
\end{equation}
To proceed, we cite a ``pre-sifted'' version of the large sieve inequality for Dirichlet characters.

\begin{lemma}
\label{lem:large_sieve}
Let $(a_n)_{n=1}^{\infty}$ be a sequence of complex numbers.  Let $P^{-}(n)=\min\{p\colon p|n\}$ if $n>1$ and $P^{-}(1)=\infty$.  If $Q\geq 1$, $U\geq 0$, and $V\geq 1$, then
\[
\sum_{q\leq Q}\Big(\log\frac{Q}{q}\Big)\sum_{\substack{\chi\pmod{q} \\ \textup{$\chi$ primitive}}}\Big|\sum_{\substack{U < n \leq U+ \lfloor V\rfloor \\ P^{-}(n)>Q}}a_n \chi(n)\Big|^2\leq (V+Q^2-1)\sum_{\substack{U < n \leq U+ \lfloor V\rfloor \\ P^-(n)>Q}}|a_n|^2.
\]
\end{lemma}
\begin{proof}
This is \cite[Theorem 9.11]{Opera}, a numerically optimal version of \cite[Lemma 4]{Gallagher}.
\end{proof}

\begin{corollary}
\label{cor:Ramare1}
Let $(a_n)_{n=1}^{\infty}$ be a sequence of complex numbers such that $\sum_{n=1}^{\infty}|a_n|$ converges.  If $Q\geq 1$ and $T\geq 1$, then
\[
\sum_{q\leq Q}\Big(\log \frac{Q}{q}\Big)\sum_{\substack{\chi\pmod{q} \\ \textup{$\chi$ primitive}}}\int_{-T}^{T}\Big|\sum_{\substack{n\geq 1 \\ P^{-}(n)>Q}}a_n \chi(n)n^{-it}\Big|^2 dt\leq 7\sum_{\substack{n\geq 1 \\ P^{-}(n)>Q}}|a_n|^2(n+Q^2\max\{T,3\}).
\]
\end{corollary}
\begin{proof}
We proceed exactly as in the proof of \cite[Corollary 6.3]{Ramare}, with the exception that we use \cref{lem:large_sieve} instead of the classical large sieve inequality (as stated as \cite[Lemma 6.1]{Ramare}).
\end{proof}

Being only slightly inefficient, we obtain from \cref{cor:Ramare1} the bound
\begin{align*}
&\sum_{q\leq Q}\sum_{\substack{\chi\pmod{q} \\ \textup{$\chi$ primitive}}}\int_{-T}^{T}\Big|\sum_{\substack{n\geq 1 \\ P^{-}(n)>Q^{1.087\,571}}}a_n \chi(n)n^{-it}\Big|^2 dt\\
&\leq \frac{11.42}{\log Q}\sum_{q\leq Q^{1.087\,571}}\Big(\log \frac{Q^{1.087\,571}}{q}\Big)\sum_{\substack{\chi\pmod{q} \\ \textup{$\chi$ primitive}}}\int_{-T}^{T}\Big|\sum_{\substack{n\geq 1 \\ P^{-}(n)>Q^{1.087\,571}}}a_n \chi(n)n^{-it}\Big|^2 dt\\
&\leq \frac{79.94}{\log Q}\sum_{\substack{n\geq 1 \\ P^{-}(n)>Q^{1.087\,571}}}|a_n|^2(n+Q^{2.175\,142}\max\{T,3\}).
\end{align*}
Fix $u\in[N_{\xi\eta},N_{\xi\eta}^{*} ]$.  By \eqref{eqn:Nchoice_large_sieve}, we have that $Q^{1.087\,571}< N_{\xi\eta}$.  Thus, if we insert $T=Q>10^4$ and let
\[
a_n = \begin{cases}
	(\log n)/n&\mbox{if $n\in[N_{\xi\eta},u]$ is prime,}\\
	0&\mbox{otherwise,}
\end{cases}
\]
then
\[
\sum_{q\leq Q}\sum_{\substack{\chi\pmod{q} \\ \textup{$\chi$ primitive}}}\int_{-Q}^{Q}\Big|\sum_{N_{\xi\eta}<p\leq u}\frac{\chi(p)\log p}{p^{1+i\tau}}\Big|^2 d\tau\leq \frac{79.94}{\log Q}\sum_{N_{\xi\eta}<p\leq u}\frac{(\log p)^2}{p}\Big(1+\frac{Q^{3.175\,142}}{p}\Big).
\]
By \eqref{eqn:Nchoice_large_sieve}, we find that if $p>N_{\xi\eta}$, then $Q^{3.175\,142}/p\leq 10^{-106}$.  It follows that
\[
\sum_{q\leq Q}\sum_{\substack{\chi\pmod{q} \\ \textup{$\chi$ primitive}}}\int_{-Q}^{Q}\Big|\sum_{N_{\xi\eta}<p\leq u}\frac{\chi(p)\log p}{p^{1+i\tau}}\Big|^2 d\tau\leq \frac{80}{\log Q}\sum_{N_{\xi\eta}<p\leq u}\frac{(\log p)^2}{p}.
\]

For $x\geq 2$, it follows from \cite[(3.24)]{RosserSchoenfeld} and \cite{Mertens} that $\log x-2\leq \sum_{p\leq x}\frac{\log p}{p}\leq \log x$.  Partial summation yields $\sum_{x<p\leq y}\frac{(\log p)^2}{p}\leq \frac{1}{2}((\log y)^2+4\log y-(\log x)^2)$, so
\begin{equation}
\label{eqn:final_primes_bound}
\sum_{q\leq Q}\sum_{\substack{\chi\pmod{q} \\ \textup{$\chi$ primitive}}}\int_{-Q}^{Q}\Big|\sum_{N_{\xi\eta}<p\leq u}\frac{\chi(p)\log p}{p^{1+i\tau}}\Big|^2 d\tau\leq \frac{40}{\log Q}((\log u)^2+4\log u-(\log N_{\xi\eta})^2).	
\end{equation}
From \eqref{eqn:Nchoice_large_sieve}, \eqref{eqn:averaged_N}, and \eqref{eqn:final_primes_bound}, we deduce the bound (recall \eqref{eqn:NchiDef})
\begin{align*}
N(1-\eta,Q)&=\sum_{q\leq Q}\sum_{\substack{\chi\pmod{q} \\ \textup{$\chi$ primitive}}}N_{\chi}(1-\eta,Q)\\
&\leq 1.613\eta^2 M_{\xi\eta}^2 e^{2.672\,32M_{\xi\eta}}\int_{N_{\xi\eta}}^{N_{\xi\eta}^{*} }\Big(\sum_{q\leq Q}\sum_{\substack{\chi\pmod{q} \\ \textup{$\chi$ primitive}}}\int_{-Q}^{Q}\Big|\sum_{N_{\xi\eta}< p\leq u}\frac{\chi(p)\log p}{p^{1+i\tau}}\Big|^2 d\tau\Big)\frac{du}{u}\\
&\leq \frac{64.52\eta^2 M_{\xi\eta}^2 e^{2.672\,32M_{\xi\eta}}}{\log Q}\int_{N_{\xi\eta}}^{N_{\xi\eta}^{*}}((\log u)^2+4\log u-(\log N_{\xi\eta})^2)\frac{du}{u}\\
&\leq 27\,034\frac{M_{\xi\eta}^5}{\eta\log Q} e^{2.672\,32M_{\xi\eta}}.
\end{align*}
Since $T=Q>10^4$, it follows from \eqref{eqn:Nchoice_large_sieve} and (vii) that $63.925\leq M_{\xi\eta}\leq 1151\eta\log Q$ and 
\[
N(1-\eta,Q)\leq 31\,116\,134 M_{\xi\eta}^4 e^{2.672\,32 M_{\xi\eta}}\leq 6.441\,034\times 10^{12}e^{2.741 M_{\xi\eta}}\leq 1.180\,016\times 10^{87}(e^{968.1455}Q^{99})^{\eta}.
\]
Since $Q=T$, we have that $\mathcal{L}=\log(QT)=2\log Q$.  Therefore, using (vii), when we choose $\sigma = 1-\eta$, we arrive at
\[
N(\sigma,Q)\leq 1.180\,016\times 10^{87}(e^{968.1455}Q^{99})^{1-\sigma},\qquad 1-\frac{1}{10A}\leq\sigma\leq 1-\frac{1}{6A\log Q}.
\]
If $1-\frac{1}{6A\log Q}<\sigma\leq 1$, then by \cref{thm:GZFR}, we have that
\[
N(\sigma,Q)\leq 1 \leq 1.180\,016\times 10^{87}(e^{968.1455}Q^{99})^{1-\sigma}.
\]
We have therefore proved that
\begin{equation}
\label{eqn:ZDE_nonexceptional}
N(\sigma,Q) \leq 1.180\,016\times 10^{87}(e^{968.1455}Q^{99})^{1-\sigma},\qquad \sigma\geq 1-\frac{1}{10A}.
\end{equation}
This is slightly better than the result stated in \cref{thm:GLFZDE}.

\subsection{Bound for $N^*(\sigma,Q)$ when $Q>10^4$}

It follows from \cref{thm:GZFR} that our claimed bound is trivial when $Q\leq 400\,000$, so we assume otherwise.  If
\[
\sigma >  1-\frac{\log (\frac{1}{(1-\beta_1(Q))(670\,564.676 + 454\,935.78 \log Q)})}{104.645 + 70.996\log Q},
\]
then by \cref{thm:DH} with $T=Q>400\,000$, we must have that  $N^*(\sigma,Q)=0$.  If
\begin{equation}
\label{eqn:exceptional_range_sigma_large_sieve}
\frac{1}{2}<\sigma\leq 1-\frac{\log (\frac{1}{(1-\beta_1(Q))(670\,564.676 + 454\,935.78\log Q)})}{104.645 + 70.996\log Q},	
\end{equation}
then upon rearrangement and applying our bound $Q>400\,000$, we must have
\begin{equation}
\label{eqn:DH_Lower}
1\leq 506\,921 (1-\beta_1(Q))(\log Q)(e^{104.645} Q^{71})^{1-\sigma}.
\end{equation}
Thus, if $\sigma\geq 1-\frac{1}{10A}$ satisfies \eqref{eqn:exceptional_range_sigma_large_sieve} and $Q>400\,000$, then by \eqref{eqn:ZDE_nonexceptional} and \eqref{eqn:DH_Lower}, $N^*(\sigma,Q)$ is at most
\[
506\,921 (1-\beta_1(Q))(\log Q)(e^{104.645} Q^{71})^{1-\sigma}N(\sigma,Q)\leq 6\times 10^{92} (1-\beta_1(Q))(\log Q)(10^{466}Q^{170})^{1-\sigma}.
\]
If $(1-\beta_1(Q))\log Q$ is at least 1, then it follows from our bound for $N(\sigma,Q)$ that we can replace $(1-\beta_1(Q))\log Q$ with $\min\{1,(1-\beta_1(Q))\log Q\}$.  Again, this is slightly better than the result stated in \cref{thm:GLFZDE}.

\section{An explicit power savings in S{\'a}rk{\"o}zy's theorem}
\label{sec:Sarkozy}

In this section, we prove \cref{thm:Sarkozy}. For ease of comparison, we will adhere as closely as possible to the setup and notation of Green \cite{Green} and make explicit his key results on zeros of Dirichlet $L$-functions.  To this end, we state the following corollary of \cref{thm:GLFZDE}.

\begin{corollary}
\label{cor:GLFZDE}
Let $Q\geq 3$ and $\sigma\geq 0$.  With $N(\sigma,Q)$ and $N^*(\sigma,Q)$ as in \eqref{eqn:NchiDef}, there holds
\[
N(\sigma,Q)\leq 10^{88}(10^{421}Q^{127})^{1-\sigma},\qquad N^*(\sigma,Q)\leq  10^{93} \min\{1,(1-\beta_1(Q))\log Q\}(10^{466}Q^{198})^{1-\sigma}.
\]
\end{corollary}
\begin{proof}
When $\sigma\geq \frac{39}{40}$, \cref{cor:GLFZDE} is weaker than \cref{thm:GLFZDE}.  When $0\leq\sigma<\frac{39}{40}$, we consider two cases.  If $Q\geq 10^{4}$, then the inflated exponents make our estimates weaker than what one can deduce from \cref{lem:basic_density} and the estimate $1-\beta_1(Q) \geq 100/(\sqrt{Q} (\log Q)^2)$ from \cref{thm:GZFR}.  If $Q<10^4$, then the total number of zeros that need to be counted is bounded above by $N(0,10^4)$.  Our upper bound for $N(\sigma,Q)$ when $0\leq\sigma<\frac{39}{40}$ and $Q<10^4$ greatly exceeds the bound for $N(0,10^4)$ that follows from \cref{lem:basic_density} and \cref{thm:GZFR} again.
\end{proof}

Define
\[
\lambda_1 = \frac{1}{20},\qquad \lambda_2 = 10^{103}, \qquad \lambda_3 = 198, \qquad M = 7 \times 3^6 \times 2^{22\,993},
\]
which correspond to explicit choices of the constants in \cite[Section 7]{Green}.  It is important to note that \cite{Green} adopts the non-standard notation of writing nontrivial zeros $\rho$ of $L(s,\chi)$ as $\beta+i\gamma = 1-\sigma+i\gamma$.  In order to avoid confusion with notation in the rest of this paper while making the comparison with \cite{Green} straightforward, we write most of the results in this section in terms of $1-\beta$ instead of $\beta$.  We record an explicit zero free region via two propositions.

\begin{proposition}
\label{prop:BG_ZFR}
	Let $Q \geq 10$ and $\mathscr{L}(s,Q)$ be as in \eqref{eqn:prod_L-functions}.  With possibly one exception, say $\rho_{*}$, every zero $\rho=\beta+i\gamma$ of $\mathscr{L}(s,Q)$ satisfies $1-\beta\geq \frac{\lambda_1}{\log Q}$. If the exceptional zero $\rho_{*}$ exists, then there exists a unique primitive quadratic character $\chi_{*}\pmod{q_{*}}$ with $q_{*}\leq Q$ such that $\rho_{*}$ is a real and simple zero of $L(s,\chi_{*})$.  The zero $\rho_{*}$, if it exists, is called the exceptional zero at scale $Q$.
\end{proposition}
\begin{proof}
	This explicit version of \cite[Proposition 7.1]{Green} follows from \cref{thm:GZFR}.
\end{proof}

Second, we record two explicit results on log-free zero density estimates.

\begin{proposition} \label{prop:BG-density-no-exception}
	Let $Q \geq 10$.  If $\alpha\geq 0$, then
	\[
	\sum_{q \leq Q} ~\sum_{\substack{\chi \pmod q \\ \textup{$\chi$ primitive}}}~\sum_{\substack{\rho=\beta+i\gamma \\  \beta \geq \alpha \\ |\gamma| \leq Q}} 1 \leq \lambda_2 Q^{\lambda_3 (1-\alpha)}.
	\]
\end{proposition}
\begin{proof}
	This explicit version of \cite[Proposition 7.4]{Green} follows from \cref{cor:GLFZDE}.
\end{proof}

\begin{proposition} \label{prop:BG-density-exception}
	Let $Q \geq 10$. If $\alpha\geq 0$ and there is an exceptional zero $\rho_{*}$ at scale $Q$, then
	\[
	\sum_{q \leq Q} ~\sum_{\substack{\chi \pmod q \\ \textup{$\chi$ primitive}}}~\sum_{\substack{\rho=\beta+i\gamma\neq\rho_* \\ \beta \geq \alpha \\ |\gamma| \leq Q}} 1 \leq \lambda_2 (1-\rho_{*})(\log Q) Q^{\lambda_3 (1-\alpha)}.
	\]
\end{proposition}
\begin{proof}
	This explicit version of \cite[Proposition 7.5]{Green} follows from \cref{cor:GLFZDE}.
\end{proof}

These preliminary results establish a helpful technical estimate. 

\begin{proposition}
\label{prop:ZDE_application}
	Set $B = 323\,905$. The following holds for any parameter $Q \geq 10$. 
	\begin{itemize}
		\item If there is no exceptional zero at scale $Q$, then
		\[
		4M \sum_{q \leq Q} ~\sum_{\substack{\chi \pmod q \\ \textup{$\chi$ primitive}}} ~\sum_{\substack{\rho=\beta+i\gamma \\  \beta \geq 1/2 \\ |\gamma| \leq Q} } Q^{-B(1-\beta)} \leq 1. 
		\]
		\item If there is an exceptional zero $\rho_{*}$ at scale $Q$, then
		\[
		Q^{-B(1-\rho_{*})} + 2M \sum_{q \leq Q} ~\sum_{\substack{\chi \pmod q \\ \textup{$\chi$ primitive}}} ~\sum_{\substack{\rho=\beta+i\gamma  \neq \rho_{*} \\ \beta \geq 1/2 \\ |\gamma| \leq Q} }Q^{-B(1-\beta)} \leq 1.
		\]
	\end{itemize}
\end{proposition}
\begin{proof}
This result parallels \cite[Proposition 7.6]{Green}.  Assume there is no exceptional zero at scale $Q$. By partial summation and \cref{prop:BG-density-no-exception}, we obtain the bound
\begin{align*}
&\sum_{q \leq Q} ~\sum_{\substack{\chi \pmod q \\ \textup{$\chi$ primitive}}} ~\sum_{\substack{\rho=\beta+i\gamma \\  \beta \geq 1/2 \\ |\gamma| \leq Q} } Q^{-B(1-\beta)}\\
&\leq \lambda_2(e^{-\lambda_1(B-\lambda_3)}-Q^{-(B-\lambda_3)/2})+\lambda_2 B(\log Q)\int_{1/2}^{1-\lambda_1/\log Q}Q^{-(B-\lambda_3)(1-\sigma)}d\sigma\\
&=\frac{\lambda_2(2B-\lambda_3)}{B-\lambda_3}(e^{-\lambda_1(B-\lambda_3)}-Q^{-(B-\lambda_3)/2})\leq 1.24 \times 10^{-6926}< 1.28\times 10^{-6926}\leq \frac{1}{4M}.
\end{align*}
If there is an exceptional zero $\rho_{*}$ at scale $Q$, then a similar calculation with \cref{prop:BG-density-exception} shows
 \begin{align*}
 Q^{-B(1-\rho_{*})} + 2M \sum_{q \leq Q} ~\sum_{\substack{\chi \pmod q \\ \textup{$\chi$ primitive}}} ~\sum_{\substack{\rho=\beta+i\gamma \neq \rho_* \\  \beta \geq 1/2 \\ |\gamma| \leq Q} }Q^{-B(1-\beta)} 
 & \leq  Q^{-B(1-\rho_{*})} + \frac{1}{2} (1-\rho_{*}) \log Q.
 \end{align*}
Note that $B$ satisfies $e^{-Bx} + \frac{1}{2} x \leq  1$ for $0 < x \leq \lambda_1 = \frac{1}{20}$. Thus, when $x = (1-\rho_{*}) \log Q$ in this exceptional case, the desired result follows. 
\end{proof}

This leads to an explicit version of a key technical result required by Green. 

\begin{proposition} \label{prop:pre-Sarkozy}
	Set $c = 9\times 10^{-14}$. Suppose that $\sigma_0 \leq \frac{1}{8}$.  Let $Q\geq 10$, and let $N\geq Q^{1/c}$ be a large integer. There exists $T\in [N^c,N^{1/120}]$, such that at least one of the following holds.
	\begin{enumerate}
		\item We have that
			\[
	\sum_{q \leq T}~\sum_{\substack{\chi \pmod q \\ \textup{$\chi$ primitive}}} ~\sum_{\substack{\rho=\beta+i\gamma \\ \frac{1}{2} \leq \beta \leq 1- \sigma_0 \\ |\gamma| \leq T}} N^{-\frac{1}{16}(1-\beta)} \leq \frac{1}{2M}.
		\]
		\item Let  $\rho_1=\beta_1+i\gamma_1$ be a zero belonging to a character with conductor $q_1 \leq T$ and $|\gamma_1| \leq T$ such that $\beta_1$ is maximal. Then $\rho_1$ is a simple, real zero of a real character $\chi_1$ of conductor $q_1 \leq T^{1/10^{6}}$, and
		\[
		N^{-\frac{1}{16}(1-\beta_1)} + 2M \sum_{q \leq T}~\sum_{\substack{\chi \pmod q \\ \textup{$\chi$ primitive}}}~\sum_{\substack{\rho=\beta+i\gamma\neq \rho_* \\1/2 \leq \beta \leq 1-\sigma_0 \\ |\gamma| \leq T} } N^{-\frac{1}{16}(1-\beta)} \leq 1.		
		\]
	\end{enumerate}
	Also, regardless of whether an exceptional zero exists, we have the bound
	\[
	\sum_{q \leq Q} ~\sum_{\substack{\chi \pmod q \\ \textup{$\chi$ primitive}}} ~\sum_{\substack{\rho=\beta+i\gamma \\  \beta \geq 1/2 \\ |\gamma| \leq T} } 1 \ll T.
	\]
\end{proposition}
\begin{proof}
	For (1) and (2), we have that $c < \min( \frac{1}{16B}, \frac{10^{-6} \lambda_1}{32 \log(4M)} ) = 9.79\ldots \times 10^{-14}$, so the proof is identical to \cite[Proposition 7.7]{Green} with our Propositions \ref{prop:BG_ZFR}-\ref{prop:ZDE_application} replacing \cite[Propositions 7.1, 7.4, 7.5, and 7.6]{Green}.  The last bound follows from Jutila's estimate  \eqref{eqn:Jutila} with $\sigma\geq 1-\frac{1}{8}$ and $\epsilon=\frac{1}{100}$.
\end{proof}

Finally, with all the necessary explicit results in hand, we may prove \cref{thm:Sarkozy}.

\begin{proof}[Proof of \cref{thm:Sarkozy}]
	In  \cite[Theorem 2.2]{Green}, Green proves the existence of constants $0 < \kappa_2 < \kappa_1$ such that if $A \subseteq \{1,\ldots, N\}$ satisfies the hypotheses of \cref{thm:Sarkozy}, then $|A| \ll N^{1-(\kappa_1-\kappa_2)}$.  By the discussion in \cite[Section 15.1]{Green}, there exists a constant $c>0$ such that if $\tau_0\in[c,\frac{1}{120}]$ and $T=N^{\tau_0}$ in \cite[Proposition 7.7]{Green}, then we may take $\kappa_1 = \frac{\tau_0}{6\times 10^4}$ and $\kappa_2 = \frac{2\tau_0}{10^6}$.  (In \cite[Section 15.1]{Green}, these are written as $\kappa_1=\frac{1}{6}c_2\tau_0$ and $\kappa_2 = 2c_3 \tau_0$, where $c_i=(\frac{1}{100})^i$.) 	Since our \cref{prop:pre-Sarkozy} makes \cite[Proposition 7.7]{Green} explicit, we may take $c =9\times 10^{-14}$, in which case
	\[
	\kappa_1-\kappa_2 = \tau_0\Big(\frac{1}{6\times 10^4}-\frac{2}{10^6}\Big) \geq \frac{9}{10^{14}}\Big(\frac{1}{6\times 10^{4}}-\frac{2}{10^6}\Big) = \frac{1.32}{10^{18}},
	\]
	as desired.
\end{proof}

\bibliographystyle{abbrv}
\bibliography{Explicit_Bombieri}

\end{document}